\documentclass{amsart}

\usepackage{latexsym,amsfonts,amsmath,amssymb,mathrsfs,url,amsthm}
\usepackage{mathtools}
\usepackage{color,graphicx}
\usepackage{lipsum}
\usepackage{hyperref}
\usepackage{xcolor}
\usepackage{placeins}
\usepackage{algorithm}
\usepackage{algpseudocode}
\usepackage{soul}
\usepackage{tikz}
\usetikzlibrary{calc}

\newtheorem{theorem}{Theorem}[section]
\newtheorem{lemma}[theorem]{Lemma}
\newtheorem{corollary}[theorem]{Corollary}
\newtheorem{proposition}[theorem]{Proposition}

\theoremstyle{definition}
\newtheorem{definition}[theorem]{Definition}

\theoremstyle{remark}
\newtheorem{remark}[theorem]{Remark}

\numberwithin{equation}{section}

\newcommand{\bbR}{\mathbb{R}}
\newcommand{\bbS}{\mathbb{S}}
\newcommand{\bbZ}{\mathbb{Z}}
\newcommand{\bsc}{\boldsymbol{c}}
\newcommand{\bsk}{\boldsymbol{k}}
\newcommand{\bsell}{\boldsymbol{\ell}}
\newcommand{\bsu}{\boldsymbol{u}}
\newcommand{\bsv}{\boldsymbol{v}}
\newcommand{\bsx}{\boldsymbol{x}}
\newcommand{\bsy}{\boldsymbol{y}}
\newcommand{\bsz}{\boldsymbol{z}}
\newcommand{\bszero}{\boldsymbol{0}}

\newcommand{\Ccal}{\mathcal{C}}
\newcommand{\Scal}{\mathcal{S}}
\newcommand{\Xcal}{\mathcal{X}}
\DeclareMathOperator{\Area}{Area}
\DeclareMathOperator{\vol}{vol}

\DeclareMathOperator{\Vor}{Vor}

\allowdisplaybreaks

\begin{document}

\title{Constructive quasi-uniform sequences over triangles}


\author{Hengjun Xu}
\address{Graduate School of Engineering, The University of Tokyo,
7-3-1 Hongo, Bunkyo-ku, Tokyo 113-8656, Japan}
\curraddr{}
\email{jokokun@g.ecc.u-tokyo.ac.jp}
\thanks{}

\author{Takashi Goda}
\address{Graduate School of Engineering, The University of Tokyo,
7-3-1 Hongo, Bunkyo-ku, Tokyo 113-8656, Japan}
\curraddr{}
\email{goda@frcer.t.u-tokyo.ac.jp}
\thanks{The work of T.G.\ is supported by JSPS KAKENHI Grant Number 23K03210.}


\date{}

\dedicatory{}

\begin{abstract}
In this paper, we develop constructive algorithms for generating quasi-uniform point sets and sequences over arbitrary two-dimensional triangular domains. Our proposed method, called the \emph{Voronoi-guided greedy packing} algorithm, iteratively selects the point farthest from the current set among a finite candidate set determined by the Voronoi diagram of the triangle. Our main theoretical result shows that, after a finite number of iterations, the mesh ratio of the generated point set is at most~2, which is known to be optimal. We further analyze two existing triangular low-discrepancy point sets and prove that their mesh ratios are uniformly bounded, thereby establishing their quasi-uniformity. Finally, through a series of numerical experiments, we demonstrate that the proposed method provides an efficient and practical strategy for generating high-quality point sets on individual triangles.
\end{abstract}

\maketitle

\sloppy

\section{Introduction}

In computational mathematics, numerical analysis, experimental design, and various engineering applications, the accuracy and efficiency of fundamental tasks such as interpolation, scattered data approximation, and physical simulation depend critically on how well the computational domain is sampled. A common approach is to distribute a carefully designed point set over the domain. To evaluate the quality of such point sets---particularly in the context of scattered data approximation---the concept of \emph{quasi-uniformity} has been central \cite{SW06,W05}. This notion is described in terms of two complementary geometric quantities: the \emph{separation radius} and the \emph{covering radius}. The separation radius is defined as half the minimum distance between any two points, whereas the covering radius is the maximum distance from any location in the domain to its nearest point. In this paper, we employ the Euclidean distance (i.e., $\ell_2$ norm); however, the notion of quasi-uniformity remains consistent for any distance metric, such as the $\ell_p$ norm with $1\le p \le \infty$, that is equivalent to the Euclidean distance.

A large separation radius prevents clustering by promoting well-separated point configurations, while a small covering radius limits the largest gaps and ensures adequate coverage of the entire domain. The \emph{mesh ratio}, defined as the covering radius divided by the separation radius, provides a unified measure of these two properties. In other words, a point set with a small mesh ratio simultaneously achieves a large separation radius and a small covering radius, and is called \emph{quasi-uniform}. A precise definition will be provided in Section~\ref{sec:preliminary}.

The importance of quasi-uniformity is directly reflected in its substantial impact on the performance of numerical algorithms. From a numerical viewpoint, the above two geometric quantities are closely tied to error control and numerical stability. The covering radius governs theoretical upper bounds on interpolation and approximation errors, whereas the separation radius affects the condition number and stability of the associated linear systems, such as interpolation matrices (see, e.g., \cite{DS10,S95,W05,WSH21}). Consequently, a quasi-uniform point set or sequence generally ensures robust performance across a wide range of applications, including radial basis function (RBF) interpolation \cite{LF03,S95}, meshfree methods \cite{SW06}, and the design of computer experiments \cite{FLS06,SWN03}. In the latter area, the \emph{maximin} and \emph{minimax} distance criteria \cite{JMY90} correspond directly to maximizing the separation radius and minimizing the covering radius, respectively; see also \cite{J16,PM12}.

In addition to quasi-uniformity, another classical but important measure of uniformity is \emph{discrepancy}. Originating from uniform distribution theory, discrepancy quantifies the deviation between the empirical distribution of a point set and the ideal uniform distribution, and plays a central role in areas such as high-dimensional numerical integration \cite{KN74,N92}. In quasi-Monte Carlo methods, low-discrepancy sequences achieve a deterministic integration error rate of $O((\log N)^d/N)$, substantially outperforming traditional Monte Carlo methods even in high dimensions.

However, these two notions of uniformity (i.e., quasi-uniformity and discrepancy) are not mutually inclusive. In contrast to quasi-uniformity, which under the Euclidean metric is defined in terms of geometric distances and is thus invariant under rigid rotations, discrepancy is typically defined with respect to a fixed family of test sets (such as axis-parallel boxes) and is generally not rotation invariant. Moreover, a point set can exhibit low discrepancy yet have a very small separation radius (i.e., a large mesh ratio), and conversely, a point set with a small mesh ratio does not necessarily have low discrepancy. Therefore, when a point set that has low discrepancy but poor separation is used as sampling nodes for kernel interpolation, it can severely degrade numerical stability. Indeed, recent studies have revealed that the relationship between these two measures is highly nontrivial. For instance, one-dimensional Kronecker sequences $i\alpha \pmod 1$, $i = 0, 1, 2, \ldots$, are quasi-uniform if and only if $\alpha$ is badly approximable \cite{G24a}, which is a stronger condition than that required for them to be low-discrepancy. As another example, the well-known Sobol' sequence, while low-discrepancy, is not quasi-uniform, at least in dimension $2$ \cite{G24b}. Further nontrivial examples are presented in \cite{DGS25}, and some positive results have been obtained for the multi-dimensional unit cube \cite{DGLPS25}, where lattice point sets and multi-dimensional Kronecker sequences are analyzed in terms of their quasi-uniformity properties. 

Despite extensive development of these theoretical and algorithmic frameworks, the construction of point sets and sequences with small or bounded mesh ratios has so far been studied mostly for highly symmetric domains, such as the unit hypercube $[0,1]^d$ \cite{DGLPS25,DGS25,JMY90,J16} or the two-dimensional sphere $\bbS^2$ \cite{HMS16}. 
Recently, Pronzato and Zhigljavsky \cite{PZ23} advanced the field by proving that, for any infinite sequence of points in arbitrary compact domains of $\bbR^d$, the mesh ratio cannot be uniformly bounded by a constant smaller than $2$, and showed that a simple greedy packing construction can provably achieve this optimal bound of $2$. However, each step of the greedy packing procedure requires identifying the point that attains the covering radius, which is computationally challenging in practice. Except for highly symmetric domains such as the unit hypercube $[0,1]^d$, the resulting sequence cannot be explicitly constructed in general. The idea of the greedy packing, i.e., sequentially selecting the point farthest from the current point set, has appeared in earlier works; see, e.g., \cite{ELPZ97} for application to progressive image sampling. 

Regarding the optimal mesh ratio of $2$, for comparison, several classical planar lattices have explicitly computable mesh ratios: the equilateral triangular lattice achieves $2/\sqrt{3}$, while the square lattice achieves $\sqrt{2}$. By contrast, for a single regular hexagon used as the sampling region, with sample points placed at its six vertices, the mesh ratio equals $2$. This finite-domain example is consistent with the universal threshold $2$ for extensible infinite sequences as proven in \cite{PZ23}, although the two settings are different. For further details on lattice configurations and their mesh ratios (referred to as the packing-covering constant therein) in higher dimensions, we refer the reader to, e.g., \cite{SV06}.

In this paper, we focus on triangular domains and study the construction of point sets and sequences with small or bounded mesh ratios. Our primary interest is in extensible constructions, namely in sequences whose initial segment of $n$ points has a controlled mesh ratio for all $n\ge 2$, rather than in the one-shot generation of a single $n$-point set. For such sequences, the aforementioned argument applies: the mesh ratio bound of $2$ is the best possible result that can be achieved. The triangle is a fundamental building block in computational science and engineering, appearing ubiquitously in areas such as mesh generation in computer graphics \cite{KCS98,SZL92}, finite element analysis \cite{BA76,C78}, and earth and space sciences \cite{SGFKKT07}. An arbitrary shaped triangle exhibits features such as asymmetry and boundary anisotropy, which make the construction and analysis of high-quality point sets more challenging than in highly symmetric domains. A canonical example arises in computational fluid dynamics, where ``skinny'' triangles with high aspect ratios are essential to resolve sharp gradients in boundary layers \cite{Ape99}, making accurate interpolation over such anisotropic elements crucial. To the best of the authors’ knowledge, the design of point sets with provably small mesh ratio bounds on arbitrary triangles, as well as the performance evaluation of low-discrepancy point sets for practical tasks like interpolation in this setting, has not yet been fully addressed in the literature.

To fill this research gap, this paper makes the following contributions:
\begin{enumerate}
    \item We propose a constructive algorithm for generating quasi-uniform point sets and sequences over arbitrary two-dimensional triangular domains, named the \emph{Voronoi-guided greedy packing} (VG) algorithm, inspired by the simple greedy packing \cite{PZ23}. We provide a theoretical analysis showing that, after finitely many iterations, the mesh ratio of the generated point sets is provably bounded by the optimal constant of $2$. 
    \item We analyze existing triangular low-discrepancy point sets proposed in \cite{BO15}, proving that their mesh ratios are uniformly bounded and thereby establishing their quasi-uniformity.
    \item We perform extensive numerical experiments to demonstrate the geometric properties of the VG algorithm and benchmark its performance against low-discrepancy point sets, random point sets, and barycentric grids in standard RBF interpolation.
\end{enumerate}
In passing, one of the triangular low-discrepancy point sets proposed in \cite{BO15} has been further extended in \cite{DHHL24,GSY17} for applications in numerical integration. 
It is also worth mentioning the related problem of generating low-discrepancy sequences on simplices, which has been studied primarily in the context of (quasi-)Monte Carlo integration. While the literature on the unit hypercube is vast, point generation on simplices often relies on mapping techniques from the hypercube (see, e.g., the monograph  \cite{FW94}, and later developments such as \cite{PC04,PC05}). Theoretical frameworks, such as the Koksma-Hlawka inequality, have also been extended to simplices \cite{BCGT13}. However, we emphasize again that low discrepancy and quasi-uniformity do not necessarily imply one another. The above second contribution implies that these two properties can coexist for the triangular low-discrepancy point sets by \cite{BO15}.

The rest of this paper is organized as follows. 
In Section~\ref{sec:preliminary}, we formally define quasi-uniformity for both infinite sequences of points and sequences of point sets with increasing size, and introduce the Voronoi diagram together with its relevant properties.
Section~\ref{sec:construction} presents the VG algorithm, a novel constructive method for generating quasi-uniform sequences over arbitrarily shaped triangles, and proves that after a finite number of iterations, the mesh ratio of the constructed point set is at most~2, which is the best possible. In the same section, we also analyze the quasi-uniformity of the existing triangular low-discrepancy point sets proposed in~\cite{BO15}.
Finally, Section~\ref{sec:experiments} provides numerical experiments, where we examine the geometric properties of the VG algorithm and compare its performance with other point sets in standard RBF interpolation problems.

\section{Preliminaries}\label{sec:preliminary}

\subsection{Quasi-uniform sequence}\label{subsec:def_qu}
Let $\Omega$ be a compact subset of $\mathbb{R}^d$ for some $ d\ge 1$, with $\vol(\Omega)\ge 0.$ We denote the norm by $\|\cdot\|$, which refers to the Euclidean norm throughout this paper. For a point set $P \subset \Omega$, the covering radius is defined as
\[  h(P; \Omega) := \max_{\bsx\in \Omega} \min_{\bsy \in P} \|\bsx-\bsy\|, \]
and the separation radius as
\[  q(P; \Omega) := \frac{1}{2}\min_{\substack{\bsx,\bsy\in P\\ \bsx\neq \bsy}} \| \bsx-\bsy\|. \]
The mesh ratio, also called the uniformity constant, is then given by
\[  \rho(P; \Omega):= \frac{h(P; \Omega)}{q(P; \Omega)}.\]
Note that at least two points are required for the separation radius to be well-defined.\footnote{The notations $h,q$, and $\rho$ for the covering radius, separation radius, and mesh ratio have been widely used in the relevant literature \cite{DS10,PZ23,S95,SW06,W05,WSH21}. Readers unfamiliar with this convention can keep in mind that the covering radius $h$ can be intuitively understood as the radius of the largest empty \emph{hole} in $\Omega$ (as we also see later in Lemma~\ref{lem:cv_radius}). Then, $q$ is reserved for the separation radius, and $\rho$ is simply the ratio.}
If $\Omega$ is connected, we always have $h(P; \Omega) \ge q(P; \Omega)$, so that $\rho(P; \Omega) \ge 1$. On the other hand, if two points in $P$ are very close to each other, the mesh ratio can become arbitrarily large.

Consider placing Euclidean balls of equal radius centered at each point in $P$. The covering radius is the minimal radius for which the union of the closed balls covers the entire domain $\Omega$, whereas the separation radius is the maximal radius such that none of the open balls overlap. From this geometric perspective, it can be shown that, for any $n$-element point set $P$, there exist constants $C_1,C_2>0$, depending only on $\Omega$, such that
\[ h(P; \Omega)\ge C_1 n^{-1/d} \quad\text{and}\quad q(P; \Omega)\le C_2 n^{-1/d}. \]
Hence, in order for $P$ to have a small mesh ratio, both the covering and separation radii should be of the same (optimal) order $n^{-1/d}$. 

Now we define quasi-uniformity for an infinite sequence of points.
\begin{definition}[quasi-uniform infinite sequence]
    Let $\Scal=(\bsx_i)_{i\ge 1}$ be an infinite sequence of points in $\Omega$. For each $n\ge 1$, denote by $P_n$ the first $n$ points of $\Scal$, i.e., $P_n=(\bsx_i)_{1\le i\le n}$. The sequence $\Scal$ is called \emph{quasi-uniform} over $\Omega$ if there exists a constant $C>0$ such that $\rho(P_n; \Omega)\le C$ for all $n\ge 2$.
\end{definition}

It was shown in \cite{DGLPS25} that, if there exists a subsequence $1<n_1< n_2 < n_3 < \cdots$ satisfying $n_{i+1}\le cn_i$ for some constant $c>1$ and $\rho(P_{n_i}; \Omega)\le C$ for all $i$, then the sequence $\Scal$ is still quasi-uniform.
Motivated by this observation, we extend the notion of quasi-uniformity to sequences of point sets that are not necessarily nested or extensible.
\begin{definition}[quasi-uniform sequence of point sets]\label{def:quasi-uniform}
    Let $(P_i)_{i\ge 1}$ be a sequence of point sets in $\Omega$. Assume that the size of each point set, denoted by $|P_i|$, satisfies $1<|P_i|<|P_{i+1}|\le c|P_i|$ for some constant $c>1$. The sequence of point sets $(P_i)_{i=0,1,\ldots}$ is called \emph{quasi-uniform} over $\Omega$ if there exists a constant $C>0$ such that $\rho(P_i; \Omega)\le C$ for all $i\ge 1$.
\end{definition}

\subsection{Voronoi diagram}

In the rest of this paper, we focus on the case where the domain $\Omega \subset \bbR^2$ is a triangle, denoted by $T=\triangle ABC$ with vertices $A$, $B$, and $C$. We assume that the triangles considered in this paper are always non-degenerate and closed (hence compact) subsets of $\bbR^2$. In what follows, we denote by $\overline{B}(\bsx,r) \subset \bbR^2$ the closed Euclidean disk of radius $r$ centered at a point $\bsx$.

\begin{definition}[largest empty disk]
Let $T = \triangle ABC$ be a triangle, and let $P \subset T$ be a finite point set.  
For a point $\bsx \in T$ and radius $r > 0$, the \emph{empty radius} at $\bsx$ is defined as
\[
r^*(\bsx; P) := \sup \{\, r > 0 \mid \overline{B}(\bsx, r) \cap P = \emptyset \,\}.
\]
Furthermore, we define
\[
r_{\max}(P;T) := \max_{\bsx \in T} r^*(\bsx; P), 
\qquad 
\bsx_{\max}(P;T) := \underset{\bsx \in T}{\arg\max}\, r^*(\bsx; P).
\]
Then, the \emph{largest empty disk} with center in $T$ is given by
\[
\overline{B}_{\max} := \overline{B}(\bsx_{\max}(P;T), r_{\max}(P;T)).
\]
\end{definition}

Note that $\bsx_{\max}(P;T)$ may not be unique; in that case, any of the maximizers can be chosen as $\bsx_{\max}(P;T)$. The same applies to the largest empty disk $\overline{B}_{\max}$. Due to the compactness of $T$, such a disk always exists. With this notion, we have the following lemma.

\begin{lemma}\label{lem:cv_radius}
Let $T = \triangle ABC$ be a triangle, and let $P \subset T$ be a finite point set. Then, it holds that $h(P; T) = r_{\max}(P;T).$
\end{lemma}

\begin{proof}
For any $\bsx\in T$ we have $r^*(\bsx;P)=\min_{\bsy\in P}\|\bsx-\bsy\|$, since the largest radius of a disk centered at $\bsx$ that does not meet $P$ is precisely the distance from $\bsx$ to the nearest point of $P$. Taking the maximum over $\bsx\in T$ on both sides yields
\[
r_{\max}(P;T)=\max_{\bsx\in T}r^*(\bsx;P)=\max_{\bsx\in T}\min_{\bsy\in P}\|\bsx-\bsy\|=h(P;T),
\]
which proves the lemma.
\end{proof}

To explicitly locate the centers of the largest empty disks, we employ the Voronoi diagram of the point set $P\subset T$. We start from the standard definition of the Voronoi diagram for a planar compact domain $\Omega\subset \bbR^2$.

\begin{definition}[Voronoi diagram] 
Let $\Omega\subset \bbR^2$ be a planar compact domain, and let $P=(\bsx_i)_{1\le i\le n}$ be a set of $n$ points in $\Omega$. The Voronoi cell corresponding to the point $\bsx_i$ is defined as
\[
\Vor(\bsx_i; P, \Omega) = \{\bsy \in \Omega \mid \|\bsx_i-\bsy\| \le \|\bsx_j-\bsy\|,\, \text{for all $j\neq i$} \}.
\]
\end{definition}

\begin{remark}
    Here, the point $\bsx_i$ is called the \emph{site} of the cell $\Vor(\bsx_i; P,\Omega)$. For $i\neq j$, the intersection of neighboring cells,  $\Vor(\bsx_i; P,\Omega)\cap \Vor(\bsx_j; P,\Omega)$, forms a \emph{Voronoi edge}, on which $d(\bsx_i,\bsy)=d(\bsx_j,\bsy)$, i.e., the points are equidistant from the two sites and the edge is contained in the bisector of the segment connecting them. Moreover, a point $\bsy$ is called a \emph{Voronoi vertex} if it lies at the intersection of at least three Voronoi edges. Such vertices are precisely the interior candidates for the centers of the largest empty disks if $\Omega=T=\triangle ABC$; boundary candidates are treated separately in Lemma~2.8. Note that the union of the Voronoi cells covers the entire domain $\Omega$.
\end{remark}

\begin{remark}
Again, our interest in this paper is in the case $\Omega = \triangle ABC$. In this setting, we define the points where the Voronoi edges meet the sides of the triangle as \emph{intersections}, and we call the union of all Voronoi edges together with the edges of the triangle itself the skeleton, 
\[
G = \left(\bigcup_{i \ne j} \Vor(\bsx_i; P,\Omega) \cap \Vor(\bsx_j; P,\Omega) \right) \cup \partial T,
\]
which captures all the edges along which the centers of the largest empty disks may lie.
\end{remark}

With the above notion of the Voronoi diagram, we can explicitly identify the possible locations of the centers of the largest empty disks $\overline{B}_{\max}$.

\begin{lemma}\label{lem:circle_center}
Let $T = \triangle ABC$ be a triangle, and let $P \subset T$ be a finite point set. Then the center of the largest empty disk in $T$ must be one of the following: a Voronoi vertex, an intersection point of a Voronoi edge with the boundary of $T$, or a vertex of $T$.
\end{lemma}

\begin{proof}
Let $\bsx_{\max} \in T$ denote the center of a largest empty disk, and let $M \ge 0$ be the number of points in $P$ that lie on its boundary.  

\begin{itemize}
\item If $M = 0$, this contradicts the definition of the largest empty disk, since the disk could be enlarged without hitting any point in $P$.
\item If $M = 1$, the disk touches exactly one point of $P$. If the center $\bsx_{\max}$ is in the interior of $T$, one could move it slightly away from that point to increase the radius, contradicting maximality. If it is on the interior of an edge, the center could be moved along the edge to enlarge the disk, unless it is pinned at a vertex. Therefore, $\bsx_{\max}$ must be a vertex of $T$ in this case.
\item If $M = 2$, the center lies on the perpendicular bisector of the two points, i.e., on a Voronoi edge $\Vor(\bsx_i; P, T) \cap \Vor(\bsx_j; P, T)$. If this edge is internal (i.e., its endpoints are Voronoi vertices), the radius can be increased by moving the center along the edge towards the endpoint with the larger radius. The maximum along this internal edge is attained only at an endpoint, which is a Voronoi vertex (the case with $M \ge 3$). Thus, $\bsx_{\max}$ cannot lie on an internal Voronoi edge while $M=2$. The only remaining possibility is that the edge intersects the boundary of $T$, and $\bsx_{\max}$ coincides with this intersection of a Voronoi edge with the boundary.
\item If $M \ge 3$, since no three points of $P$ that lie on the boundary of the disk can be collinear (otherwise no circle passes through them), the center is uniquely determined as the point equidistant from these $M$ points. This is exactly a Voronoi vertex of $P$.
\end{itemize}
No other possibilities exist: $\bsx_{\max}$ cannot lie in the interior of a Voronoi cell corresponding to a single site (it could be moved to increase the radius), nor in the interior of an edge without touching 2 points. 
Therefore, the center of the largest empty disk lies either at a Voronoi vertex, an intersection of a Voronoi edge with the boundary, or a vertex of $T$, as claimed.
\end{proof}

To prepare for our proposed algorithm in the next section, we study the geometric properties of the point set consisting of the three vertices of a triangle.

\begin{lemma}\label{lem:initial}
Let $T= \triangle ABC$ be a triangle with edge lengths $a\ge b\ge c$, and let $P=\{A,B,C\}$ be the set of its three vertices. Then
\[
q(P;T) = \frac{c}{2}, \quad
h(P;T) = 
\begin{cases} 
ab^2/(a^2+b^2-c^2) & \text{if $T$ is obtuse}, \\[1mm]
R & \text{otherwise},
\end{cases}
\]
where $R$ denotes the radius of the circumcircle of $T$.
\end{lemma}

\begin{proof}
The separation radius \(q(P;T)\) is defined as half of the minimal distance between points in \(P\). Since the shortest edge of the triangle has length \(c\), we immediately have $q(P;T)=c/2$.

For the covering radius $h(P;T)$, consider the largest empty disk problem with sites $P=\{A,B,C\}$. By Lemma~\ref{lem:circle_center}, the center of any largest empty disk must be either a Voronoi vertex (an intersection of perpendicular bisectors), an intersection of a Voronoi edge with the boundary $\partial T$, or a vertex of $T$. However, in the present case, $P$ coincides with the set of vertices of $T$, so the last possibility can be excluded.

The unique Voronoi vertex for three noncollinear points is the circumcenter of $T$, whose radius is $R$. If the circumcenter lies inside $T$ (this happens exactly when $T$ is acute or right), then it is a valid center and yields $h(P;T)=R$.

If the circumcenter lies outside $T$ (equivalently, if $T$ is obtuse), then the circumcenter is not an admissible center. In this case, the remaining candidates are the intersection points of the perpendicular bisectors with the boundary $\partial T$; we denote this set by $Q$, which consists of at most six points. Hence 
\[ h(P;T)=\max_{\bsx\in Q}\min_{\bsy\in \{A,B,C\}}\|\bsx-\bsy\|.\] 
Without loss of generality, we place the triangle as
\[
B=(0,0),\qquad C=(a,0),\qquad A=(x_A,y_A),
\]
with
\[
x_A=\frac{a^2-b^2+c^2}{2a},\qquad
y_A=\sqrt{c^2-x_A^2}>0,
\]
so that $\|BC\|=a,\ \|CA\|=b,\ \|AB\|=c$. 
A straightforward calculation shows that the maximum in the covering radius is attained at the intersection point $D$ between the perpendicular bisector of $CA$ and the opposite edge $BC$. This intersection has coordinates
\[ D=\left( \frac{a(a^2-c^2)}{a^2+b^2-c^2}, 0 \right),\]
and the corresponding minimum distance to $P=\{A,B,C\}$ is
\[ \|DA\| = \|DC\| = \frac{ab^2}{a^2+b^2-c^2}.\]
This completes the proof.
\end{proof}

As a direct consequence of the previous lemma, we can establish an upper bound on the mesh ratio for the vertex set of a triangle:

\begin{corollary}\label{cor:mesh_ratio_vertices}
Let $T = \triangle ABC$ be a triangle, and let $P = \{A, B, C\}$ be the set of its three vertices. Then
\[
\rho(P;T) \le \frac{1}{\sin \theta_{\min}},
\]
where $\theta_{\min}$ denotes the smallest interior angle of $T$.
\end{corollary}

\begin{proof}
Let $c$ denote the length of the side opposite $\theta_{\min}$ (so $c$ is the shortest side). By the previous lemma, the covering radius $h(P;T)$ takes two possible forms depending on the triangle type: If $T$ is acute or right, we have $\rho(P;T)=2R/c$. By the sine rule, it holds that $c = 2R \sin \theta_{\min}$, and thus
\[
\rho(P;T) = \frac{1}{\sin \theta_{\min}}.
\]
If $T$ is obtuse, we have $a^2>b^2+c^2$ with $a$ being the longest side, giving
\[ h(P;T)=\frac{ab^2}{a^2+b^2-c^2}\le \frac{ab^2}{(b^2+c^2)+b^2-c^2}=\frac{a}{2}.\]
Using the sine rule again, it holds that $a = 2R \sin A$ and $c = 2R \sin \theta_{\min}$, so
\[
\rho(P;T) \le \frac{a}{c}= \frac{\sin A}{\sin \theta_{\min}} \le \frac{1}{\sin \theta_{\min}}.
\]
Combining both cases completes the proof.
\end{proof}

\section{Constructive algorithms}\label{sec:construction}

In this section, we present constructive algorithms for generating quasi-uniform point sequences over triangular domains.
Our primary interest is in \emph{extensible} constructions, that is, infinite sequences whose initial segments of $n$ points have a bounded mesh ratio for all $n\ge 2$, rather than in the one-shot generation of a single $n$-point set.
Our main contribution is the proposal of the \emph{Voronoi-guided greedy packing} (VG) algorithm, which iteratively places points based on geometric information derived from the Voronoi diagram of the current point set.
The VG algorithm guarantees that, after a finite number of iterations, the mesh ratio of the constructed point set is at most~2, which is the best possible.
We further analyze the quasi-uniformity of the existing triangular low-discrepancy point sets proposed in~\cite{BO15}, providing a theoretical comparison with the proposed method.

\subsection{Voronoi‐guided greedy packing (VG) algorithm}

To introduce our VG algorithm, we start from the simple greedy packing algorithm proposed in~\cite{PZ23}, which in general form proceeds as shown in Algorithm~\ref{alg:greedy}. It iteratively adds the point that attains the covering radius, i.e., the point farthest from the existing points, to the current point set. Since, for a general domain $\Omega$, it is computationally hard to identify that point, the algorithm is not fully constructive.

\begin{algorithm}[t]
  \caption{Simple greedy packing}\label{alg:greedy}
  \begin{algorithmic}[1]
    \Require Compact subset $\Omega \subset \mathbb{R}^d$
    \State Initialize $P_1 \gets \{\bsx_1\}$ with $\bsx_1\in \Omega$
    \For{$n = 1,2,\dots$}
      \State Select the next point
        $\bsx_{n+1} \gets 
           \displaystyle \underset{\bsx \in \Omega}{\arg\max}
             \min_{1\le i \le n}\|\bsx - \bsx_i\|$
      \State Update the point set
        $P_{n+1} \gets P_n \cup \{\bsx_{n+1}\}$
    \EndFor
  \end{algorithmic}
\end{algorithm}

Our VG algorithm can be regarded as a constructive version of this method. As proven in Lemma~\ref{lem:circle_center}, the point attaining the covering radius must belong to a set of finite candidates, allowing it to be identified with finite computational cost. 
This way, Algorithm~\ref{alg:greedy} can be replaced by a computationally feasible algorithm for triangle domains.
Algorithm~\ref{alg:VG} presents the full procedure of the VG algorithm.

\begin{algorithm}[t]
  \caption{Voronoi-guided greedy packing (VG)}\label{alg:VG}
  \begin{algorithmic}[1]
    \Require Triangle domain $T = \triangle ABC \subset \bbR^2$
    \State Initialize $P_3 = \{\bsx_1,\bsx_2,\bsx_3\} = \{A,B,C\}$
    \For{$n = 3,4,5,\dots$}
      \State Compute the Voronoi diagram of $P_n$ in $T$
      \State Let $\Ccal_n$ be the set of candidate points:
        \Statex \hspace{2em} (i) all Voronoi vertices inside $T$
        \Statex \hspace{2em} (ii) all intersections of Voronoi edges with the boundary $\partial T$
      \State Select the next point
        $
          \bsx_{n+1} \gets 
          \displaystyle \underset{\bsx \in \Ccal_n}{\arg\max} \;\min_{1\le i \le n}\|\bsx - \bsx_i\|
        $
      \State Update the point set: $P_{n+1} \gets P_n \cup \{\bsx_{n+1}\}$
    \EndFor
  \end{algorithmic}
\end{algorithm}

Lines~1 and 2 in Algorithm~\ref{alg:VG} can be replaced by ``Initialize $P_1=\{\bsx_1\}$ with $\bsx_1\in T$'' and ``\textbf{for $n=1,2,3,\ldots$ do}'', respectively. In this case, the vertices of $T$ should also be included in the set $\Ccal_n$ in line~4 (cf.\ Lemma~\ref{lem:circle_center}). Then, the resulting sequence is identical to that generated by the simple greedy packing algorithm (Algorithm~\ref{alg:greedy}), up to local ordering differences when the $\arg\max$ in line~5 is not unique.
Furthermore, if the first three points generated by Algorithm~\ref{alg:greedy} coincide with the three vertices of $T$, then both algorithms will produce the same subsequent sequence of points. This follows directly from Lemma~\ref{lem:circle_center}. Such a scenario occurs, for example, when $T$ is an equilateral triangle and the initial point for the simple greedy packing algorithm is chosen as one of its vertices. 
Then, it follows from \cite{PZ23} that the mesh ratio of the generated point set is bounded by $2$ for all $n$.

Nevertheless, we choose to initialize the algorithm with the three-vertex configuration. This choice is conceptually and computationally advantageous for several reasons. First, as already noted, we do not need to include the vertices of $T$ in the candidate set $\Ccal_n$. Second, the three vertices define the target domain, so including them from the outset ensures that the generated point set explicitly captures the boundary. Third, incorporating the corners from the beginning helps distribute subsequent points more evenly across the interior, preventing any corners from being omitted.

As observed in Corollary~\ref{cor:mesh_ratio_vertices}, the mesh ratio for the initial three-vertex configuration cannot always be strictly bounded by 2, particularly when the smallest angle $\theta_{\min}$ satisfies $\theta_{\min} < \pi/6$. This highlights a limitation compared to the uniform mesh-ratio bound of $2$ given in \cite{PZ23}. Interestingly, the choice of our initialization, i.e., the three-vertex configuration, becomes less significant as the sequence grows. We will show that, after finitely many iterations, the mesh ratio of the point set constructed by the VG algorithm is guaranteed to be at most 2. This result further enriches the theoretical understanding of the greedy packing approach (cf.\ Remark~\ref{rem:extension}).

\begin{remark}[Computational cost]
In the VG algorithm, each iteration requires the Voronoi diagram of the current point set $P_n$ in $T$. Using standard planar algorithms such as the sweepline method \cite{F87,BCKO08}, this can be computed in $O(n\log n)$ time. The next point is chosen from the candidate set consisting of Voronoi vertices inside $T$ and intersections of Voronoi edges with the boundary $\partial T$, whose cardinality is $O(n)$. Once the Voronoi diagram is available, the quantity $\min_{1\le i\le n}\|\bsx-\bsx_i\|$ for each candidate $\bsx$ is determined by the local Voronoi structure. Hence, selecting the next point requires only an additional $O(n)$ scan over the candidate set. Therefore, if the Voronoi diagram is recomputed from scratch at each iteration, generating the first $N$ points requires $O(N^2\log N)$ time in total.

In an incremental implementation, the Voronoi diagram can be updated locally after inserting the new point (see, e.g., \cite{LZ06} for local updating policies). Furthermore, by maintaining the candidate points and their corresponding distances in a priority queue (such as a max-heap), the next furthest point can be extracted in $O(\log n)$ time. Since the generated point sets are proven later to be quasi-uniform (cf.~Theorem~\ref{thm:main}), the local degree of each vertex in the dual Delaunay triangulation is bounded by a constant. This guarantees that only $O(1)$ Voronoi vertices are inserted or deleted during the local update. Consequently, maintaining the data structures takes $O(\log n)$ time per iteration, which optimally reduces the overall time complexity for generating $N$ points to $O(N \log N)$.
\end{remark}

Figure~\ref{fig:example_VG} shows an example of the point sets generated by the VG algorithm for $n=3$ to $n=11$. As observed, the points are distributed quasi-uniformly over the triangular domain, with each new point placed farthest from the existing points according to the Voronoi-guided procedure.

\begin{figure}[t]
    \centering
    \includegraphics[width=\linewidth]{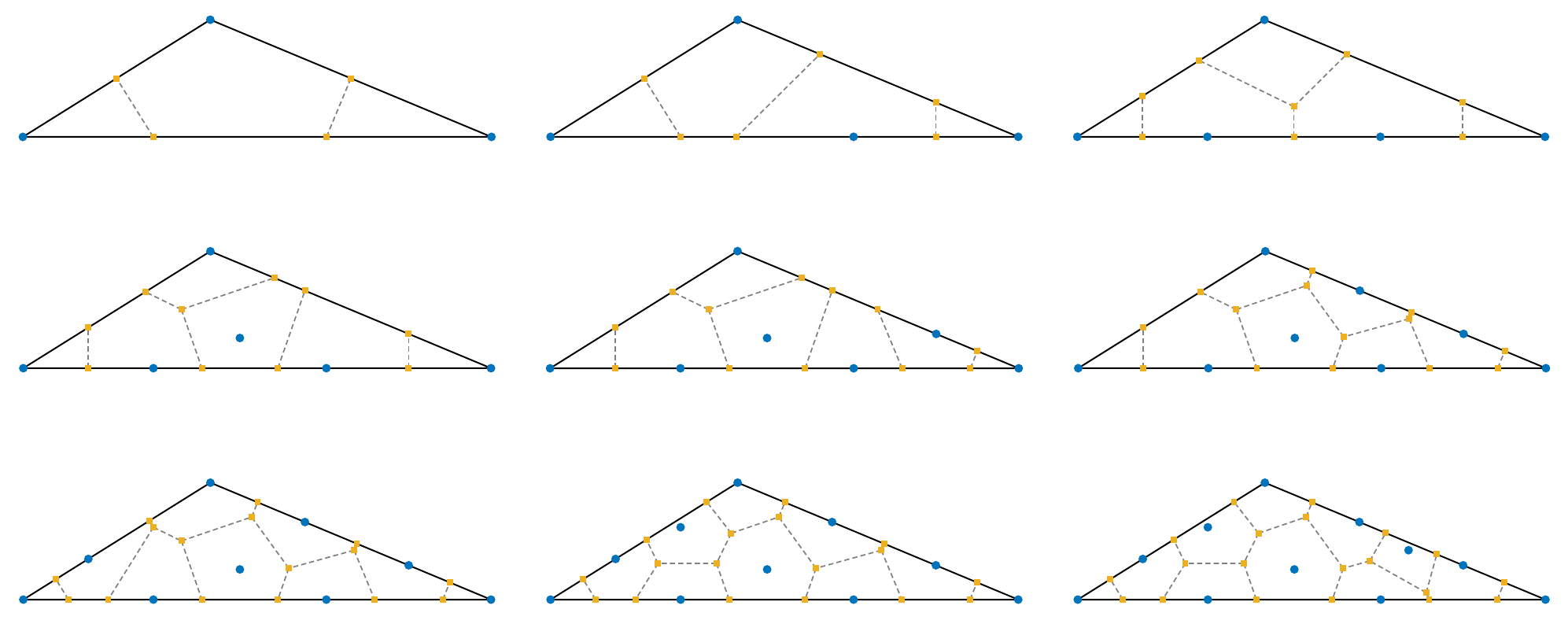}
    \caption{Example of point sets generated by the VG algorithm from $n=3$ to $n=11$. In each subplot, the blue dots denote the generated points, the dashed lines denote the Voronoi edges, and the orange dots denote the candidate set.}
    \label{fig:example_VG}
\end{figure}

\subsubsection{Quasi-uniformity for the VG algorithm}

Here, we study the quasi-uniformity of the point set generated by the VG algorithm. 
We first study the dynamics of the VG algorithm. In fact, the second item in the following lemma is stronger than what is required for proving the main result. Nevertheless, we include the full statement, as it is of independent interest.

\begin{lemma}\label{lem:dynamics}
Let $T = \triangle ABC$ be a triangle, and let $(\bsx_i)_{i\ge 1}$ be the sequence of points generated by Algorithm~\ref{alg:VG}. Denote $P_n=\{\bsx_1,\dots,\bsx_n\}$.
Then the following statements hold.
\begin{enumerate}
    \item For any $n\ge 3$,  
    \[
    q(P_{n+1}; T) = \begin{cases}
        q(P_n;T) & \text{if }\rho(P_n;T)>2,\\
        h(P_n;T)/2 & \text{otherwise.}
    \end{cases}
    \]
    \item For any $n \ge 3$, we have $h(P_{n+1}; T) \le h(P_n; T)$. Moreover, let 
    \[ \Xcal_{n}:= \underset{\bsx \in T}{\arg\max} \;\min_{1\le i \le n}\|\bsx - \bsx_i\|\]
    be the set of points that attain the covering radius; its cardinality is finite and denoted by $|\Xcal_n|$. Then, after inserting all points of $\Xcal_n$ (in any order), we have a strict inequality
    \[ h(P_{n+|\Xcal_n|}; T) < h(P_n; T). \]
    \item If $\rho(P_n; T) > 2$ for some $n\ge 3$, then $\|\bsx_i - \bsx_j\| \ge 2q(P_3; T)$ for any $1\le i<j\le n$.
\end{enumerate}
\end{lemma}

\begin{proof}
    We prove each item in turn. Throughout the proof, we write $d_{P}(\bsy)=\min_{\bsx \in P}\|\bsx-\bsy\|$.

    Regarding the first item, it follows from  Lemma~\ref{lem:circle_center} that the point $\bsx_{n+1}$ satisfies $d_{P_n}(\bsx_{n+1})=h(P_n;T)$. Thus, for the separation radius of $P_{n+1}=P_n\cup\{\bsx_{n+1}\}$, it holds that
    \begin{align*}
        q(P_{n+1};T) & = \frac{1}{2}\min\left\{\min_{1\le i<j\le n}\|\bsx_i-\bsx_j\|,\ \min_{1\le i\le n}\|\bsx_{n+1}-\bsx_i\|\right\} \\
        & = \frac{1}{2}\min\{2q(P_n;T),\, d_{P_n}(\bsx_{n+1})\}\\
        & = \frac{1}{2}\min\{2q(P_n;T),\, h(P_n;T)\}\\
        & = \begin{cases}
            q(P_n;T) & \text{if $\rho(P_n;T)>2$,}\\
            h(P_n;T)/2 & \text{otherwise.}
        \end{cases}.
    \end{align*}
    This proves the claim.
    
    Regarding the second item, fix $\bsy\in T$ and a point set $P$. For any $\bsx\in T$ we have
    \[
    d_{P\cup\{\bsx\}}(\bsy)
    = \min\{d_P(\bsy),\|\bsx-\bsy\|\} \le d_P(\bsy).
    \]
    Taking the supremum over $\bsy\in T$ yields the non-increasing property
    \[
    h(P\cup\{\bsx\};T)=\max_{\bsy\in T}  d_{P\cup\{\bsx\}}(\bsy) \le \max_{\bsy\in T}  d_P(\bsy)=h(P;T).
    \]
    Substituting $P=P_n$ and $\bsx=\bsx_{n+1}$ proves $h(P_{n+1};T)\le h(P_n;T)$.

    It remains to prove the strict decrease after inserting all maximizers. By Lemma~\ref{lem:circle_center}, the function $d_{P_n}(\cdot)=\min_{1\le i\le n}\|\cdot-\bsx_i\|$ attains its maximum only at a finite set of Voronoi vertices and at a finite number of intersections of Voronoi edges with $\partial T$. Hence $\Xcal_n$ is finite. 
    
    For each $\bsx \in \Xcal_n$ we have $d_{P_n}(\bsx)=h(P_n;T)$. If we insert $\bsx$ into the point set, then $d_{P_n\cup\{\bsx\}}(\bsx)=0 < d_{P_n}(\bsx)$, while $d_{P_n\cup\{\bsx\}}(\bsy)\le d_{P_n}(\bsy)$ for any $\bsy\in T$. Thus, inserting any single $\bsx \in\Xcal_n$ strictly reduces the value of $d_{P_n}$ at this particular maximizer, but other maximizers in $\Xcal_n$ may still attain the old maximum. If we insert all points of $\Xcal_n$ (in any order) to obtain $P_{n+|\Xcal_n|} = P_n\cup\Xcal_n$, then every maximizer $\bsx\in\Xcal_n$ satisfies $d_{P_{n+|\Xcal_n|}}(\bsx)=0 < h(P_n;T)$. All other points $\bsy \in T\setminus \Xcal_n$ already had $d_{P_n}(\bsy) < h(P_n;T)$. Hence
    \[
        \max_{\bsy\in T} d_{P_{n+|\Xcal_n|}}(\bsy) < h(P_n;T),
    \]
    i.e., $h(P_{n+|\Xcal_n|};T)<h(P_n;T)$, as claimed.
    
    Finally, we prove the third item. It follows from the first two items that, if $\rho(P_n; T) > 2$, we have
    \[ \rho(P_{n+1}; T)=\frac{h(P_{n+1}; T)}{q(P_{n+1}; T)}\le \frac{h(P_n; T)}{q(P_n; T)}=\rho(P_n; T), \]
    implying the non-increasing property of the mesh ratio. 
    Thus, we know that $\rho(P_3;T)\geq \rho(P_4;T)\geq \cdots \geq \rho(P_n;T)>2$, and for any $1\le i<j\le n$ with $j\ge 4$, we obtain
    \begin{align*}
        \|\bsx_i - \bsx_j\| & \ge d_{P_{j-1}}(\bsx_j) = h(P_{j-1};T) = \rho(P_{j-1};T) q(P_{j-1};T) \\
        & > 2q(P_{j-1};T)=2q(P_3;T),
    \end{align*} 
    where the last equality follows from the first item of this lemma.
    If $j\le 3$, we simply have $\|\bsx_i - \bsx_j\|\ge 2q(P_3;T)$ by the definition of the separation radius. This completes the proof.
\end{proof}

We now state the main theorem of this paper, which gives an explicit upper bound on the number of points required for the mesh ratio to reach the optimal value of 2.

\begin{theorem}\label{thm:main}
Let $T = \triangle ABC$ be a triangle, and let $(\bsx_i)_{i\ge 1}$ be the sequence of points generated by Algorithm~\ref{alg:VG}. Denote $P_n=\{\bsx_1,\dots,\bsx_n\}$.

\begin{enumerate}
\item If $\rho(P_3;T)\leq 2$, then $\rho(P_n;T)\leq 2$ for all $n\ge 4$. 
\item If $\rho(P_3;T)> 2$, then $\rho(P_n;T)\leq 2$ for all $n> K$, where
\[
K \le \left\lfloor \frac{A_T+L_T\, q(P_3;T)+\pi\, q(P_3;T)^2}{\pi\, q(P_3;T)^2}\right\rfloor ,
\]
and $A_T$ and $L_T$ are the area and perimeter of $T$, respectively.
\end{enumerate}
\end{theorem}

Note that, in the second item above, since $\{\bsx_1,\bsx_2,\bsx_3\}=\{A,B,C\}$, Lemma~\ref{lem:initial} implies that $q(P_3;T)$ equals half of the shortest edge length of $T$. Hence, the bound on $K$ is completely determined by the shape of $T$. More precisely, if $T= \triangle ABC$ has the edge lengths $a\ge b\ge c$, the bound on $K$ is equal to
\[
\left\lfloor 
\sqrt{\left(\frac{a+b}{c}+1\right)
\left(\frac{a+b}{c}-1\right)
\left(1+\frac{a-b}{c}\right)
\left(1+\frac{b-a}{c}\right)}
+\frac{2}{\pi}\left(\frac{a+b}{c}+1\right)
+1
\right\rfloor .
\]

\begin{proof}[Proof of Theorem~\ref{thm:main}]
    Suppose that there exists $n\ge 3$ such that $\rho(P_n;T)\leq 2$. Then, by the first two items of Lemma~\ref{lem:dynamics}, we have
    \[ \rho(P_{n+1};T) = \frac{h(P_{n+1};T)}{q(P_{n+1};T)}\le \frac{h(P_{n};T)}{h(P_{n};T)/2}=2. \]
    This means that $\rho(P_{n+k};T)\leq 2$ for all $k\ge 0$, proving the first item.

    If $\rho(P_n;T)> 2$, this implies that any pairwise distance $\|\bsx_i-\bsx_j\|$ is larger than or equal to $2q(P_3;T)$ for $1\le i<j\le n$. Consider placing open Euclidean disks of equal radius $q(P_3;T)$ centered at each point in $P_n$. Then, these disks are disjoint and contained in the enlarged domain given by the Minkowski sum:
    \[ T \oplus B(\bszero,q(P_3;T)):=\left\{ \bsx+\bsy\in \bbR^2\mid \bsx\in T,\, \bsy\in B(\bszero,q(P_3;T))\right\} , \]
    where $B(\bszero,q(P_3;T))$ denotes the Euclidean disk of radius $q(P_3;T)$ centered at the origin. By comparing the area of the union of the $n$ disks with that of $T \oplus B(\bszero,q(P_3;T))$, we obtain
    \[ n \pi (q(P_3;T))^2\le \Area\,(T \oplus B(\bszero,q(P_3;T))) = A_T+L_T\, q(P_3;T)+\pi\, q(P_3;T)^2, \]
    which leads to
    \[ n\leq \frac{A_T+L_T\, q(P_3;T)+\pi\, q(P_3;T)^2}{\pi (q(P_3;T))^2}.\]
    For larger $n$, we must have $\rho(P_n;T)\le 2$, completing the proof of the second item.
\end{proof}

\begin{remark}\label{rem:extension}
Following the argument used to prove Lemma~\ref{lem:dynamics} and Theorem~\ref{thm:main}, the result of \cite{PZ23} on the uniform bound of the mesh ratio can be extended as follows. 
For any compact domain $\Omega \subset \mathbb{R}^d$ and any initial point set $P \subset \Omega$ of arbitrary size $|P| \ge 2$, the simple greedy packing algorithm can generate a sequence of points whose mesh ratio is bounded above by $2$ after an additional $K$ iterations, where
\[
K = \frac{\vol(\Omega \oplus B(\bszero, q(P;\Omega)))}{\vol(B(\bszero, q(P;\Omega)))},
\]
with $B(\bszero, q(P;\Omega))$ being the $d$-dimensional Euclidean ball of radius $q(P;\Omega)$ centered at the origin.
\end{remark}

As an additional property of the point set generated by the VG algorithm, the following result is of independent interest.
\begin{proposition}\label{prop:good-sub}
Let $T = \triangle ABC$ be a triangle, and let $(\bsx_i)_{i\ge 1}$ be the sequence of points generated by Algorithm~\ref{alg:VG}.
For $n \ge 3$, denote $P_n = \{\bsx_1, \dots, \bsx_n\}$.

Consider a Voronoi vertex $\bsc \in \Ccal_n$, and let $\bsx_i, \bsx_j, \bsx_k \in P_n$ be the three points whose Voronoi cells meet at $\bsc$.
Denote by $T_{ijk} = \triangle \bsx_i \bsx_j \bsx_k$ the triangle formed by these three points.
Then,
\begin{align*}
\min\left\{
\sin \angle_{\bsx_i}(T_{ijk}),
\sin \angle_{\bsx_j}(T_{ijk}),
\sin \angle_{\bsx_k}(T_{ijk})
\right\}
\ge \frac{1}{\rho(P_n; T)},
\end{align*}
where $\angle_{\bsx_i}(T_{ijk})$ denotes the interior angle of $T_{ijk}$ at vertex $\bsx_i$.
\end{proposition}

\begin{proof}
Since $\bsc \in \Ccal_n$ is the circumcenter of the triangle $T_{ijk}$, its circumradius $R(T_{ijk})$ satisfies
\[
R(T_{ijk}) = \|\bsx_i - \bsc\| = \|\bsx_j - \bsc\| = \|\bsx_k - \bsc\| \le h(P_n;T),
\]
where the inequality follows from the definition of the covering radius.
Moreover, by the definition of the separation radius, we have $\|\bsx_i-\bsx_j\|,\|\bsx_j-\bsx_k\|,\|\bsx_k-\bsx_i\| \ge 2q(P_n;T)$. Using the law of sines, e.g., $\|\bsx_i-\bsx_j\| = 2R(T_{ijk})\sin \angle_{\bsx_k}(T_{ijk})$, and hence
\[ \sin \angle_{\bsx_k}(T_{ijk}) \ge \frac{2q(P_n;T)}{2R(T_{ijk})}=\frac{q(P_n;T)}{R(T_{ijk})}\ge \frac{q(P_n;T)}{h(P_n;T)}=\frac{1}{\rho(P_n;T)}.\]
The same argument applies to $\sin \angle_{\bsx_i}(T_{ijk})$ and $\sin \angle_{\bsx_j}(T_{ijk})$, completing the proof.
\end{proof}

\begin{remark}
    As can be seen from Theorem~\ref{thm:main} and Proposition~\ref{prop:good-sub}, after $K$ iterations, every Voronoi vertex in $\Ccal_n$ is surrounded by a triangle formed by three points of $P_n$, whose interior angles are all greater than or equal to $\arcsin(1/2) = \pi/6$.
\end{remark}

Moreover, although our primary motivation arises from applications to RBF interpolation, the VG algorithm naturally induces the following approximation scheme based on piecewise constant functions. Given a function $f: T \to \mathbb{R}$ and a point set $P_n = \{\bsx_1, \dots, \bsx_n\} \subset T$ generated by Algorithm~\ref{alg:VG}, consider the piecewise constant function
\[ g(\bsx) = f(\bsx_i), \quad \text{for $\bsx\in \Vor(\bsx_i; P_n,T)$}. \]
When the assignment $\bsx \in \Vor(\bsx_i; P_n,T)$ is not unique (e.g., when $\bsx$ lies on a Voronoi edge or vertex), the choice of $\bsx_i$ can be made arbitrarily. If $f$ is H\"{o}lder continuous, then we obtain
\begin{align*}
    \left(\frac{1}{\Area\,(T)}\int_{T}\left| f(\bsx)-g(\bsx)\right|^2\, \mathrm{d} \bsx\right)^{1/2} & \le \max_{\bsx}\left| f(\bsx)-g(\bsx)\right|\\
    & = \max_{i=1,\ldots,n}\max_{\bsx\in \Vor(\bsx_i; P_n,T)}\left| f(\bsx)-f(\bsx_i)\right|\\
    & \le C \, \max_{i=1,\ldots,n}\max_{\bsx\in \Vor(\bsx_i; P_n,T)}\left\| \bsx-\bsx_i\right\|^{\alpha}\\
    & = C \left( h(P_n;T)\right)^{\alpha}.
\end{align*}
Since $\rho(P_n; T)$ is uniformly bounded, there exists a constant $C_h$ such that $h(P_n; T) \le C_h n^{-1/2}$.
Consequently, for H\"{o}lder classes, the piecewise constant approximation achieves the optimal convergence rate for both the $L_2$ and $L_\infty$ errors, namely of order $n^{-\alpha/2}$, see, e.g., \cite[Section~1.3.9]{N88}.

\subsection{Triangular low-discrepancy point sets}

In \cite{BO15}, two types of such point sets were proposed: the \emph{triangular van der Corput sequence} and the \emph{triangular Kronecker lattices}. The former is a digital construction that generates points hierarchically as the centroids of nested subtriangles determined by the base-4 representation of non-negative integers. The latter is constructed by scaling the integer lattice $\bbZ^2$, rotating it by an angle whose tangent is a quadratic irrational number, and applying an affine map to fit the triangular domain. The algorithms for generating these point sets in a triangular domain are given in Algorithms~\ref{alg:vdC} and~\ref{alg:Kronecker}, respectively.
Although it was sufficient to restrict the discrepancy analysis in \cite{BO15} to equilateral triangles, this is not the case for the quasi-uniformity analysis, and hence our subsequent analysis directly deals with triangles of arbitrary shape. 

We stress that the scope of the constructions in this subsection is essentially two-dimensional.
While the quasi-uniformity framework in Section~\ref{sec:preliminary} as well as the greedy-packing principle in Remark~\ref{rem:extension} apply to general compact domains in $\mathbb{R}^d$, the explicit constructions studied here are specific to planar triangles. In particular, the triangular van der Corput sequence relies on a recursive subdivision into four congruent subtriangles, which does not admit a straightforward analogue for higher-dimensional simplices.
Moreover, unlike the VG algorithm, the triangular Kronecker lattice is naturally indexed by the target size $N$ and should therefore be viewed primarily as a benchmark for finite point sets rather than as an extensible construction. Accordingly, the purpose of this subsection is to study the mesh-ratio behavior of these existing constructions, establishing the fact that the low-discrepancy property and quasi-uniformity can coexist.

\begin{algorithm}[t]
\caption{Triangular van der Corput sequence}
\label{alg:vdC}
\begin{algorithmic}[1]
    \Require Triangle domain $T = \triangle ABC \subset \bbR^2$
    \State $P \gets \emptyset$
    \For{$n = 0,1,2,\dots$}
        \State $T_n \gets T$
        \State $i \gets n$
        \While{$i > 0$}
            \State $(A', B', C') \gets \text{vertices of $T_n$}$
            \State $d \gets i \pmod 4$
            \State $T_n \gets \begin{cases}
                ((B'+C')/2, (C'+A')/2, (A'+B')/2) & \text{if $d=0$,}\\
                (A', (A'+B')/2, (A'+C')/2) & \text{if $d=1$,}\\
                (B', (B'+C')/2, (B'+A')/2) & \text{if $d=2$,}\\
                (C', (C'+A')/2, (C'+B')/2) & \text{if $d=3$}
            \end{cases}$
            \State $i \gets \lfloor i/4 \rfloor$
        \EndWhile
        \State $(A', B', C') \gets \text{vertices of } T_n$
        \State $\bsx_n \gets (A'+B'+C')/3$
        \State $P \gets P \cup \{\bsx_n\}$
    \EndFor
\end{algorithmic}
\end{algorithm}

\begin{algorithm}[t]
  \caption{Triangular Kronecker lattice} \label{alg:Kronecker}
  \begin{algorithmic}[1]
    \Require Triangle domain $T = \triangle ABC \subset \bbR^2$, number of points $N$, rotation angle $\alpha\in [0,2\pi)$
    \State $R \gets \triangle ((0,0), (0,1), (1,0))$
    \State $n \gets \lceil \sqrt{2N}\rceil + 1$
    \State $P \gets \{-n,-n+1,\dots,n-1,n\}^2$
    \ForAll{$\bsx \in P$}
      \State $\bsx \gets\dfrac{1}{\sqrt{2N}}
        \begin{pmatrix}
          \cos\alpha & -\sin\alpha\\
          \sin\alpha &  \cos\alpha
        \end{pmatrix} \bsx$
    \EndFor
    \State Remove all points from $P$ that lie outside $R$
    \State (Optional) Add/remove points to make $|P| = N$
    \ForAll{$\bsx=(x_1,x_2) \in P$}
      \State $\bsx \gets A + x_1(C-A) + x_2(B-A)$ \Comment{Affine map to target triangle $T$}
    \EndFor
  \end{algorithmic}
\end{algorithm}

\subsubsection{Quasi-uniformity of triangular van der Corput sequence}

First, we consider four congruent sub-triangles of $T$ and study the covering and separation radii of the point set consisting of their centroids. 

\begin{lemma}\label{lem:vdC_covering}
Let $T=\triangle ABC$ be a triangle, and let $P=\{O,P_A,P_B,P_C\}$ denote the set consisting of the centroid $O$ of the central sub-triangle and the three centroids $P_A,P_B,P_C$ of the corner sub-triangles obtained by connecting the midpoints of the three sides. Let $m_a\le m_b\le m_c$ be the lengths of the three medians of $T$, and define $m_{\max}:=\max\{m_a,m_b,m_c\}=m_c$. Then, we have
\[ h(P;T)=\frac{m_{\max}}{3}.\]
\end{lemma}

\begin{proof}
Split $T$ at the midpoints of its sides to obtain four sub-triangles with centroids $O$, $P_A$, $P_B$, $P_C$.  
For any point $\bsx$ in a sub-triangle, its distance to the corresponding centroid is at most one-third of the maximal median of that sub-triangle, which is bounded by $m_{\max}/3$.  
Since the four sub-triangles partition $T$, every point $\bsx\in T$ lies within distance $m_{\max}/3$ of some centroid in $P$.  
Thus, the union of the four closed balls of radius $m_{\max}/3$ centered at $O$, $P_A$, $P_B$, $P_C$ covers $T$, leading to $h(P;T)\leq m_{\max}/3$.

The lower bound $h(P;T)\ge m_{\max}/3$ is clear by considering the vertex $C$, to which $P_C$ is the nearest among $P$, and its distance is exactly $m_{\max}/3$. 
\end{proof}

\begin{lemma}\label{lem:vdC_separation}
Let $T=\triangle ABC$ be a triangle with side lengths $a,b,c$. Let $P=\{O,P_A,P_B,P_C\}$ and $m_a\le m_b\le m_c$ be defined as in Lemma~\ref{lem:vdC_covering}. Moreover, let $m_{\min}:=\min\{m_a,m_b,m_c\}=m_a$ and $c_{\min}:=\min\{a,b,c\}$. Then, we have
\[
q(P;T) = \min\left\{ \frac{m_{\min}}{6}, \frac{c_{\min}}{4} \right\}.
\]
\end{lemma}

\begin{figure}[tbp]
    \centering
    \begin{tikzpicture}[scale=1.1, >=stealth]
        \coordinate (B) at (0,0);
        \coordinate (C) at (5,0);
        \coordinate (A) at (1.5,3.5);

        \coordinate (Ma) at ($(B)!0.5!(C)$);
        \coordinate (Mb) at ($(A)!0.5!(C)$);
        \coordinate (Mc) at ($(A)!0.5!(B)$);

        \coordinate (O)  at (barycentric cs:A=1,B=1,C=1);
        \coordinate (Pa) at (barycentric cs:A=4,B=1,C=1);
        \coordinate (Pb) at (barycentric cs:A=1,B=4,C=1);
        \coordinate (Pc) at (barycentric cs:A=1,B=1,C=4);

        \draw[thick] (A) -- (B) -- (C) -- cycle;

        \draw[dashed, gray] (Ma) -- (Mb) -- (Mc) -- cycle;

        \draw[dotted, thick, gray] (A) -- (Ma);
        \node[gray, left] at ($(A)!0.75!(Ma)$) {\small $m_a$};

        \draw[very thick] (O) -- (Pa);
        \node[right] at ($(O)!0.75!(Pa)$) {\small $m_a/3$};

        \draw[very thick] (Pb) -- (Pc);
        \node[above] at ($(Pb)!0.6!(Pc)$) {\small $a/2$};

        \fill (A) circle (1.5pt) node[above] {$A$};
        \fill (B) circle (1.5pt) node[left] {$B$};
        \fill (C) circle (1.5pt) node[right] {$C$};

        \fill (O) circle (1.5pt) node[above right, inner sep=2pt] {$O$};
        \fill (Pa) circle (1.5pt) node[above right, inner sep=2pt] {$P_A$};
        \fill (Pb) circle (1.5pt) node[above left, inner sep=2pt] {$P_B$};
        \fill (Pc) circle (1.5pt) node[above right, inner sep=2pt] {$P_C$};
    \end{tikzpicture}
    \caption{Geometric configuration of the centroids $O, P_A, P_B, P_C$ and their relevant distances as discussed in Lemma~\ref{lem:vdC_separation}.}
    \label{fig:vdC_separation}
\end{figure}

\begin{proof}
(See Figure~\ref{fig:vdC_separation} for an illustration of the geometric configuration.) In this proof, indices are taken modulo~$3$, i.e., we write $(A,B,C) = (A_0,A_1,A_2)$ with $A_{i+3}=A_i$. The distance from the central centroid $O$ to a corner centroid $P_i$ is one-third of the median issued from $A_i$, i.e., $\|O P_i\| = m_i/3$. Hence, the minimal distance involving $O$ is $m_{\min}/3$.
The distance between two corner centroids $P_i$ and $P_{i+1}$ is simply half of the side connecting the corresponding vertices, i.e., $\|P_i P_{i+1}\| = \|A_i A_{i+1}\|/2$. Thus, the minimal distance between two corner centroids is $c_{\min}/2$. Since the separation radius $q(P;T)$ is defined as half of the minimal pairwise distance among all centroids, we obtain
\[
q(P;T) = \frac{1}{2} \min \Bigl\{ \min_i \|O P_i\|, \min_i \|P_i P_{i+1}\| \Bigr\} = \min\left\{ \frac{m_{\min}}{6}, \frac{c_{\min}}{4} \right\},
\]
as claimed.
\end{proof}

We are now ready to prove the quasi-uniformity of the triangular van der Corput sequence.

\begin{theorem}\label{thm:vdC}
Let $T=\triangle ABC$ be a triangle with side lengths $a,b,c$, and let $m_a, m_b, m_c$ denote the lengths of its three medians. Define
\[
    m_{\max}:=\max\{m_a,m_b,m_c\}, \quad
    m_{\min}:=\min\{m_a,m_b,m_c\}, \quad
    c_{\min}:=\min\{a,b,c\}.
\]
Let $(\bsx_i)_{i\ge 0}$ be the sequence of points generated by Algorithm~\ref{alg:vdC}, and denote $P_n=\{\bsx_0,\dots,\bsx_{n-1}\}$. Then, for any $n\ge 4$, it holds that
\[
    \rho(P_n;T) \le 4m_{\max}\min\left\{ \frac{1}{m_{\min}}, \frac{2}{3c_{\min}} \right\}.
\]
\end{theorem}

\begin{proof}
Let $k$ be the unique positive integer such that $4^k \le n < 4^{k+1}$. 
We say that the level $k$ is \emph{filled} when $n = 4^k$; otherwise, the level $(k+1)$ is being inserted.

The key observation is that when level $k$ is filled, the point set $P_{4^k}$ consists exactly of the centroids of $4^k$ congruent sub-triangles of $T$ \cite{BO15}. 
Hence, thanks to the recursive structure in which each sub-triangle is further subdivided into four congruent sub-triangles, and by Lemma~\ref{lem:vdC_covering}, we obtain
\[
    h(P_{4^k};T) 
    = \frac{h(P_{4^{k-1}};T)}{2} 
    = \cdots 
    = \frac{h(P_4;T)}{2^{k-1}}
    = \frac{m_{\max}}{3\cdot 2^{k-1}}.
\]
Since the covering radius is non-increasing with respect to $n$ (see the proof of Lemma~\ref{lem:dynamics}), it follows that for any $4^k \le n < 4^{k+1}$,
\[
    h(P_n;T) \le h(P_{4^k};T) = \frac{m_{\max}}{3\cdot 2^{k-1}}.
\]

Similarly, once level $k$ is filled, Lemma~\ref{lem:vdC_separation} yields
\[
    q(P_{4^k};T) 
    = \frac{q(P_{4^{k-1}};T)}{2} 
    = \cdots 
    = \frac{q(P_4;T)}{2^{k-1}}
    = \frac{1}{2^{k-1}} \min\left\{ \frac{m_{\min}}{6}, \frac{c_{\min}}{4} \right\}.
\]
Therefore, for any $4^k \le n < 4^{k+1}$, we have
\[
    q(P_n;T) \ge q(P_{4^{k+1}};T)
    = \frac{1}{2^k} \min\left\{ \frac{m_{\min}}{6}, \frac{c_{\min}}{4} \right\}.
\]
Combining these inequalities completes the proof.
\end{proof}

\begin{remark}
    As observed from the above proof, when $n = 4^k$, the following equality holds:
    \[
    \rho(P_{4^{k}};T) = 2m_{\max}\min\left\{ \frac{1}{m_{\min}}, \frac{2}{3c_{\min}} \right\}.
    \]
    Hence, if this value already exceeds $2$, the triangular van der Corput sequence cannot achieve the optimal uniform bound on the mesh ratio. In particular, for an equilateral triangle, the recursive subdivision into four congruent subtriangles preserves exact self-similarity at every level. In this highly symmetric case, the digital construction is well adapted to the geometry, leading to the optimal mesh ratio of $2/\sqrt{3}$ along the subsequence of $n=4^k$.
\end{remark}

\subsubsection{Quasi-uniformity of triangular Kronecker lattices}

Here, we study the quasi-uniformity of triangular Kronecker lattices. In \cite{BO15}, it was shown that if the rotation angle $\alpha$ is chosen so that $\tan(\alpha)$ is a quadratic irrational number, the point set generated by Algorithm~\ref{alg:Kronecker} achieves a discrepancy of order $(\log N)/N$. However, for the purpose of quasi-uniformity, the specific choice of $\alpha$ is not essential. In what follows, we ignore the optional step in line~8 of Algorithm~\ref{alg:Kronecker}, since adding points without care can make the separation radius arbitrarily small.

\begin{remark}
A key feature of Algorithm~\ref{alg:Kronecker} is that the point set is first constructed in the reference triangle $R$ and then mapped to the target triangle $T$ via an affine transformation of the form $F(\bsx)=M\bsx+\boldsymbol{b}$.
Geometrically, the linear part $M$ maps Euclidean circles in $R$ to ellipses in $T$. For any vector $\bsu,\bsv\in R$, the properties of the operator norm $\|\cdot\|_{\mathrm{op}}$ imply that
\[ \frac{\|\bsu - \bsv\|}{\|M^{-1}\|_{\mathrm{op}}} \le \|M(\bsu - \bsv)\| \le \|M\|_{\mathrm{op}}\|\bsu - \bsv\|.
\]
This means the covering radius $h$ is expanded by at most $\|M\|_{\mathrm{op}}$, while the separation radius $q$ is compressed by at most $1/\|M^{-1}\|_{\mathrm{op}}$. Consequently, the mesh ratio is amplified by the factor
\[
\kappa(M) := \|M\|_{\mathrm{op}}\|M^{-1}\|_{\mathrm{op}},
\]
which is the condition number of $M$ (see \cite{Tur48} for the foundational concept). In this sense, $\kappa(M)$ measures the anisotropy, or ``ellipticality'', of the deformation from the reference geometry to the target domain, a property that is explicitly captured in the upper bound in Theorem~\ref{thm:Kronecker}.
\end{remark}

To prove the quasi-uniformity, we shall use the following elementary result.

\begin{lemma}\label{lem:T_contains_integer}
    Let $T\subset \bbR^2$ be a isosceles right triangle whose two legs have length $\sqrt{2}+1$ (and thus the hypotenuse has length $2+\sqrt{2}$). Then, $T$ contains at least one integer lattice point. Moreover, for any point $\bsx\in T$, the distance to the nearest integer lattice point contained in $T$ is bounded above by $\sqrt{2+\sqrt{2}}+1/\sqrt{2}$.
\end{lemma}

\begin{proof}
    Consider the incircle $S$ of the isosceles right triangle $T$. One can easily check that its radius is given by $1/\sqrt{2}$. Since the covering radius of the integer lattice $\bbZ^2$ in $\bbR^2$ is also $1/\sqrt{2}$, the incircle $S$ must contain at least one integer lattice point. Therefore, the triangle $T$ itself contains at least one integer lattice point. 
    
    For any point $\bsx\in T$, the distance from $\bsx$ to the center of $S$ is bounded by $\sqrt{2+\sqrt{2}}$. Furthermore, from the center of $S$, there exists an integer lattice point $\bsz\in S$ within a distance of $1/\sqrt{2}$. The desired bound then follows immediately from the triangle inequality.
\end{proof}

\begin{theorem}\label{thm:Kronecker}
Let $T=\triangle ABC$ be a triangle, and let $P$ be the point set generated by Algorithm~\ref{alg:Kronecker}. 
Let $F:R\to T$ denote the affine map from the reference triangle $R$ to $T$, and let $M$ be its linear part.
Then, for any $N \ge 2$ and any rotation angle $\alpha$ there holds
\[
\rho(P;T)\le C_T:=\left( 2\sqrt{2+\sqrt{2}}+3\sqrt{2}+2\right) \kappa(M),
\]
where $\kappa(M)$ is the condition number of the matrix $M$. In particular, the bound $C_T>0$ depends only on the geometry of $T$.
\end{theorem}

\begin{proof}
Let $\Lambda_n = \{-n,\dots,n\}^2$ with $n = \lfloor \sqrt{2N} \rfloor + 1$, and let 
\[
P' = \{ S_N R_\alpha \bsk \mid \bsk \in \Lambda_n \}
\] 
be the rotated and scaled lattice, where $R_\alpha$ is the rotation matrix and $S_N = 1/\sqrt{2N}$ is the scaling factor.  
Let $R$ be the reference isosceles right triangle $\triangle ((0,0), (0,1), (1,0))$, and let $P_R = P' \cap R$ be the subset lying in $R$. 
Define the affine map $F: R \to T$ by
\[
F(x_1,x_2) := A + x_1 (C-A) + x_2 (B-A),
\] 
and let $P = F(P_R)$, which is exactly our point set. Furthermore, let $M$ denote the linear part of $F$.

First, consider the separation radius. For distinct $\bsx, \bsy \in P_R$, we have $\bsx - \bsy = S_N R_\alpha (\bsk - \bsell)$ for distinct $\bsk,\bsell \in \Lambda$. Since $R_\alpha$ preserves the Euclidean norm and $\|\bsk - \bsell\| \ge 1$, it follows that
$\|\bsx - \bsy\| \ge S_N =  1/\sqrt{2N}$.
Applying the affine map $F$, we get
\[
\|F(\bsx)-F(\bsy)\| = \|M(\bsx-\bsy)\| \ge \frac{\|\bsx-\bsy\|}{\|M^{-1}\|_{\mathrm{op}}},
\] 
where $\|M^{-1}\|_{\mathrm{op}}$ is the operator $2$-norm (or the spectral norm) of $M^{-1}$, depending only on the geometry of the triangle $T=\triangle ABC$.
Hence, the separation radius of $P$ satisfies
\[
q(P; T) \ge \frac{1}{2\|M^{-1}\|_{\mathrm{op}}\, \sqrt{2N}}.
\]

Next, we consider the covering radius. 
For any fixed $\bsy \in T$, let $\bsx = F^{-1}(\bsy)\in R$. Then it follows from the definition of $h(P_R;R)$ that there exists at least one $\bsz\in P_R$ such that $\|\bsx-\bsz\|\le h(P_R;R)$. 
By considering the corresponding point $F(\bsz)\in T$, we have
\[ \|\bsy-F(\bsz)\| = \|M(\bsx-\bsz)\|\le \|M\|_{\mathrm{op}}\|\bsx-\bsz\|\le \|M\|_{\mathrm{op}}\, h(P_R;R). \]
Thus, it suffices to prove an upper bound on $h(P_R;R)$. We have
\begin{align*}
    h(P_R;R) & = \max_{\bsx\in R} \min_{\bsy\in P_R}\|\bsx-\bsy\| \\
    & = \max_{\bsx\in (S_NR_{\alpha})^{-1}R}\min_{\bsy\in (S_NR_{\alpha})^{-1}P_R}\|S_{N}R_{\alpha}(\bsx-\bsy)\|\\
    & = \frac{1}{\sqrt{2N}}\max_{\bsx\in (S_NR_{\alpha})^{-1}R}\min_{\bsy\in \bbZ^2\cap (S_NR_{\alpha})^{-1}R}\|\bsx-\bsy\|,
\end{align*}
where $(S_NR_{\alpha})^{-1}R := \{ (S_NR_{\alpha})^{-1}\bsx\mid \bsx\in R\}$ is a rotated and enlarged isosceles right triangle with leg length $\sqrt{2N}>\sqrt{2}+1$. 

We now define
\[ u = \left\lfloor \frac{\sqrt{2N}}{\sqrt{2}+1}\right\rfloor\ge 1 \quad \text{and}\quad  v = \sqrt{2N}-(\sqrt{2}+1)u. \]
Then we have $0\le v< \sqrt{2}+1$. Let $(\sqrt{2}+1)u\, R_{\alpha}^{-1}R$ be the rotated and enlarged isosceles right triangle with leg length $(\sqrt{2}+1)u$. Since $(\sqrt{2}+1)u\le \sqrt{2N}$, it follows that $(\sqrt{2}+1)u\, R_{\alpha}^{-1}R\subseteq (S_NR_{\alpha})^{-1}R$, and $(\sqrt{2}+1)u\, R_{\alpha}^{-1}R$ can be expressed as the union of $\sum_{k=1}^{u}(2k-1)=u^2$ sub-triangles, all of which are isosceles right triangles with leg length $\sqrt{2}+1$. (See Figure~\ref{fig:u_v_discussion} for an illustration of this geometric relationship.) 

\begin{figure}[tbp]
    \centering
    \begin{tikzpicture}[scale=0.55, >=stealth]
        \def\u{5}
        \def\step{1.5}
        \def\v{0.8}

        \coordinate (O) at (0,0);
        \coordinate (X) at (\u*\step + \v, 0);
        \coordinate (Y) at (0, \u*\step + \v);

        \coordinate (Xi) at (\u*\step, 0);
        \coordinate (Yi) at (0, \u*\step);

        \draw[very thick] (O) -- (X) -- (Y) -- cycle;

        \draw[thick, gray] (O) -- (Xi) -- (Yi) -- cycle;

        \foreach \i in {1,...,\u} {
            \draw[dashed, gray] (\i*\step, 0) -- (\i*\step, {\u*\step - \i*\step});
            \draw[dashed, gray] (0, \i*\step) -- ({\u*\step - \i*\step}, \i*\step);
            \draw[dashed, gray] (0, \i*\step) -- (\i*\step, 0);
        }

        \draw[<->] (0, -0.3) -- (\u*\step, -0.3) node[midway, below] {\small $(\sqrt{2}+1)u$};
        \draw[<->] (\u*\step, -0.3) -- (\u*\step + \v, -0.3) node[midway, below] {\small $v$};

        \node[black, align=center] at (\u*\step*0.37, \u*\step*0.25) {$u^2$ sub-triangles};
        
    \end{tikzpicture}
    \caption{Geometric relationship between $u$ and $v$ in the proof of Theorem~\ref{thm:Kronecker}. The large outer right triangle represents the domain with leg length $\sqrt{2N}=(\sqrt{2}+1)u+v$, which is partitioned into $u^2$ smaller right triangles and the remaining region.}
    \label{fig:u_v_discussion}
\end{figure}

It then follows from Lemma~\ref{lem:T_contains_integer} that, if $\bsx\in (\sqrt{2}+1)u\, R_{\alpha}^{-1}R$, there exists at least one integer lattice point in $(\sqrt{2}+1)u\, R_{\alpha}^{-1}R$ within a distance of $\sqrt{2+\sqrt{2}}+1/\sqrt{2}$. Otherwise, if $\bsx\in (S_NR_{\alpha})^{-1}R\setminus (\sqrt{2}+1)u\, R_{\alpha}^{-1}R$, there exists a point $\bsy\in (\sqrt{2}+1)u\, R_{\alpha}^{-1}R$ within a distance of $v$, from which there exists at least one integer lattice point in $(\sqrt{2}+1)u\, R_{\alpha}^{-1}R$ within a distance of $\sqrt{2+\sqrt{2}}+1/\sqrt{2}$. 
This implies that
\[ h(P_R;R) \le \frac{1}{\sqrt{2N}}\left( \sqrt{2+\sqrt{2}}+1/\sqrt{2}+v\right)\le \frac{1}{\sqrt{2N}}\left( \sqrt{2+\sqrt{2}}+3/\sqrt{2}+1\right).\]
and hence
\[ h(P;T)\le \frac{\|M\|_{\mathrm{op}}}{\sqrt{2N}}\left( \sqrt{2+\sqrt{2}}+3/\sqrt{2}+1\right).\]

Combining the bounds on the separation and covering radii, we obtain
\begin{align*}
    \rho(P;T) & \leq \left( 2\sqrt{2+\sqrt{2}}+3\sqrt{2}+2\right) \|M\|_{\mathrm{op}}\|M^{-1}\|_{\mathrm{op}}\\
    & = \left( 2\sqrt{2+\sqrt{2}}+3\sqrt{2}+2\right) \kappa(M).
\end{align*}
Since this bound is independent of $N$ and $\alpha$, we complete the proof.
\end{proof}

\begin{remark}
    Although triangular Kronecker lattices are not, in general, extensible, they are quasi-uniform as a sequence of point sets with increasing sizes (cf.~Definition~\ref{def:quasi-uniform}). Again, in contrast to \cite{BO15}, no specific choice of the rotation angle~$\alpha$ is required to ensure quasi-uniformity. However, since the condition number satisfies $\kappa(M)\ge 1$ for any affine transformation matrix~$M$, our upper bound on the mesh ratio is always greater than the optimal constant, which is $2$. Whether this bound can be improved---possibly by selecting a suitable $\alpha$---remains an open question for future research. 
\end{remark}

\section{Numerical experiments}\label{sec:experiments}

To evaluate the practical performance of various point sets, we conduct a series of numerical experiments.
First, we examine the geometric sensitivity of the proposed VG algorithm with respect to the shape of the underlying triangle.
Next, we analyze and compare the behavior of the mesh ratio for different point sets as the number of points increases.
Finally, we assess their performance in a radial basis function (RBF) interpolation task.

\subsection{Geometric sensitivity of the VG algorithm}

We study how the shape of a triangle affects the mesh ratio. To eliminate scale effects, the triangle shape is quantified using the isoperimetric quotient $J = 12\sqrt{3}\,A/L^2$, where $A$ and $L$ denote the area and perimeter, respectively. This dimensionless, similarity-invariant index is normalized such that $J = 1$ for an equilateral triangle. Triangles are generated via a parameter sweep over the angles $(\alpha,\beta)$, with $\gamma = \pi - \alpha - \beta > 0$, and side lengths set as $a:b:c = \sin\alpha:\sin\beta:\sin\gamma$ under the constraint $L = 1$. Coordinates are fixed by placing side $c$ along the $x$-axis. For each triangle $T$, the point configuration is generated by the VG algorithm.

\begin{figure}[t]
   \centering
    \includegraphics[width=\linewidth]{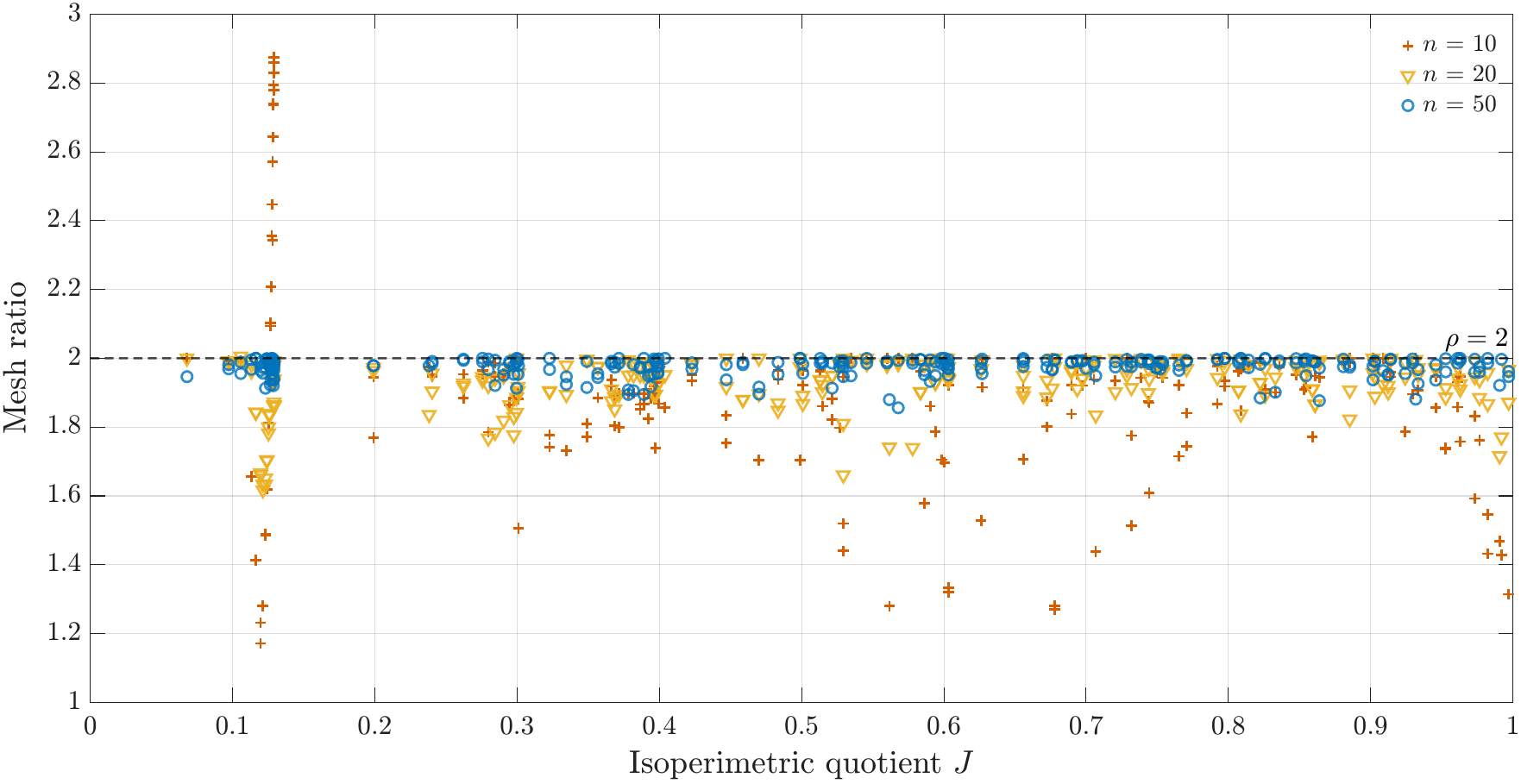}
    \caption{Mesh ratio of point sets generated by the VG algorithm for various isoperimetric quotients $J$, shown for $n=10$ (red), $n=20$ (orange), and $n=50$ (blue).}
    \label{fig:4-1}
\end{figure}

Figure~\ref{fig:4-1} shows a scatter plot of the mesh ratio $\rho$ versus $J$ for $n \in \{10, 20, 50\}$. Most configurations lie near the optimal bound of $\rho \approx 2$. When $n=10$, near-equilateral triangles ($J \approx 1$) achieve the lowest values, with $\rho$ approaching the lower limit $2/\sqrt{3}$ in some cases. As the shape becomes more degenerate ($J \ll 1$), many triangles produce significantly larger $\rho$, reflecting that the separation radius $q$ is pinned by the shortest side ($q = c/2$) while the covering radius $h$ remains large. In fact, extremely skinny triangles occasionally yield outliers with even larger $\rho$. Increasing $n$ primarily reduces $h$ (improving coverage) and hence lowers $\rho$, but the reduction is modest when $q$ is constrained by a short edge, highlighting the strong geometric sensitivity. Nevertheless, as the number of points increases, the mesh ratio for even nearly degenerate, skinny triangles tends to converge back toward the baseline of $\rho = 2$, corroborating Theorem~\ref{thm:main}.

\begin{figure}[t]
   \centering
    \includegraphics[width=\linewidth]{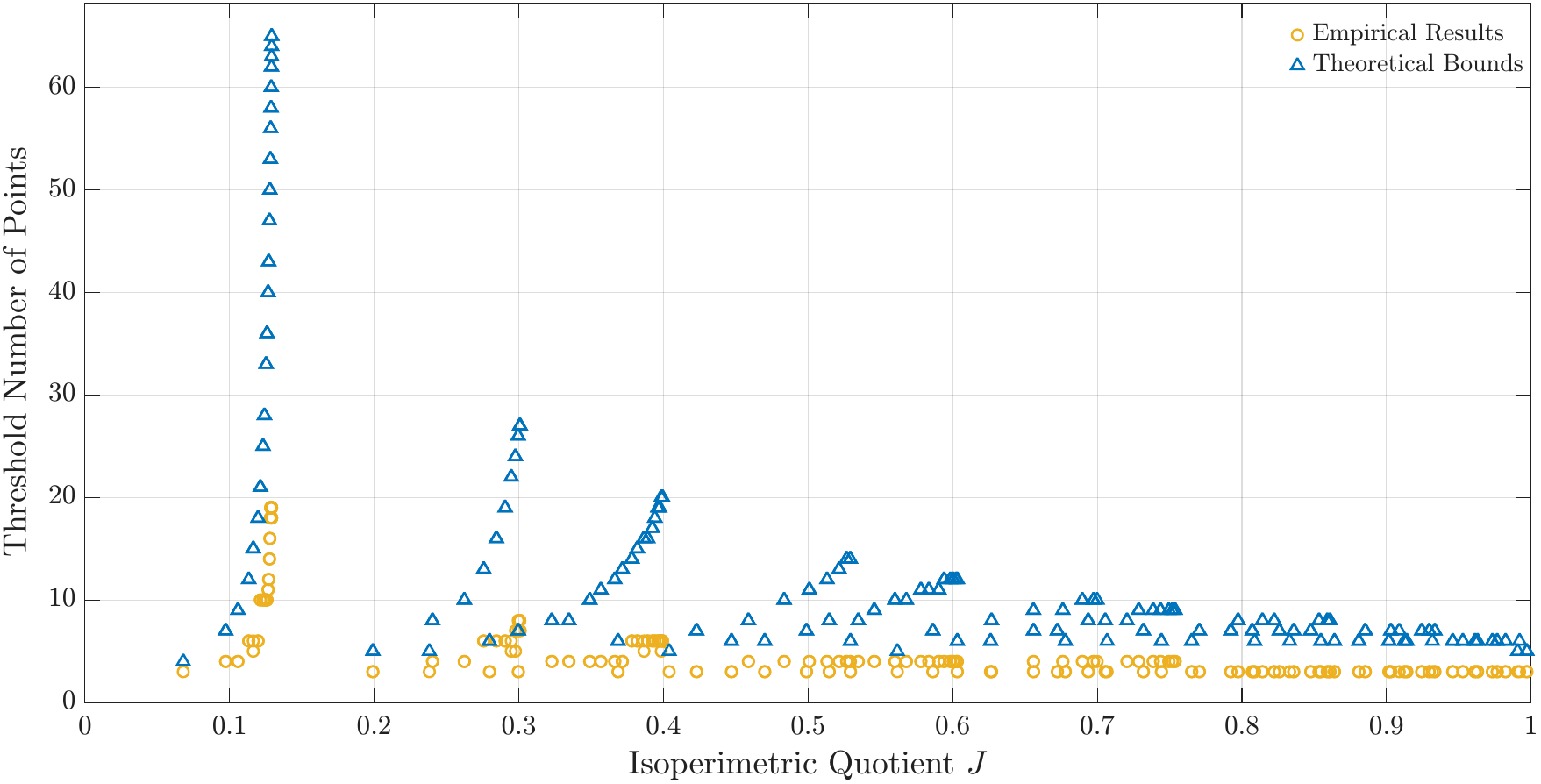}
    \caption{Number of points required for the VG algorithm to achieve the optimal mesh ratio of $2$, comparing empirical results (orange) with theoretical bounds (blue).}
    \label{fig:4-2}
\end{figure}

We know from Theorem~\ref{thm:main} that the number of points required for the VG algorithm to achieve the optimal mesh ratio of $2$ is at most
\[ \left\lfloor \frac{A_T+L_T\, q(P_3;T)+\pi\, q(P_3;T)^2}{\pi\, q(P_3;T)^2}\right\rfloor + 1. \]
We compare this theoretical bound with the empirical result. The empirical result is defined as the first number of points $n$ for which the mesh ratio of the point set generated by the VG algorithm is less than or equal to $2$. This comparison is shown in Figure~\ref{fig:4-2}.
As illustrated in the figure, the results vary significantly with the isoperimetric quotient $J$.
For well-shaped triangles (e.g., $J \ge 0.8$), the initial three-vertex configuration already satisfies the optimal mesh ratio.
For intermediate values ($J \approx 0.55$), the theoretical bound suggests a significantly larger number of points, whereas the empirical result remains small, typically just $3$ or $4$.
For poorly-shaped triangles (i.e., as $J$ decreases further), both the theoretical bound and the empirical result increase sharply, confirming that poorer geometry requires significantly more points to achieve the optimal mesh ratio.
Overall, this experiment supports that $J$ is strongly correlated with the number of points (and thus the computational effort) required for the VG algorithm to reach the optimal mesh ratio when initialized with the three-vertex configuration.
Furthermore, the theoretical bound is consistently larger than the empirical result, indicating that the bound is not necessarily tight. Improving this bound remains open for future work.

\subsection{Comparison of mesh ratio}

Here, six different point sets are compared: our VG algorithm (Algorithm~\ref{alg:VG}), the triangular van der Corput sequence (Algorithm~\ref{alg:vdC}), the triangular Kronecker lattice (Algorithm~\ref{alg:Kronecker} with $\alpha=3\pi/8$), the barycentric grid, and two random point sets. Note that the angle $\alpha=3\pi/8$ for the Kronecker lattice is chosen to maintain consistency with the numerical experiments presented in \cite{BO15}. 

Our barycentric grid point set is generated using a hybrid approach. First, we construct a uniform barycentric lattice with $m$ divisions per side, where $m$ is chosen to create the largest possible complete lattice containing no more than $n$ points. If the lattice size is smaller than $n$, we then apply a farthest point insertion algorithm to iteratively add the remaining points until the total set size reaches exactly $n$.

To provide random baselines, we introduce two types of random point sets. The first is a classical independent and identically distributed (i.i.d.) uniform random set. The second is a Poisson-disk–like set (PD) \cite{DH06}, produced via a sequential inhibition process to achieve a blue-noise distribution that avoids point clustering, which is often used in computer graphics. For the mesh ratio analysis, results for these random sets are averaged over 100 independent trials for each sample size.

\begin{figure}[t]
   \centering
    \includegraphics[width=\linewidth]{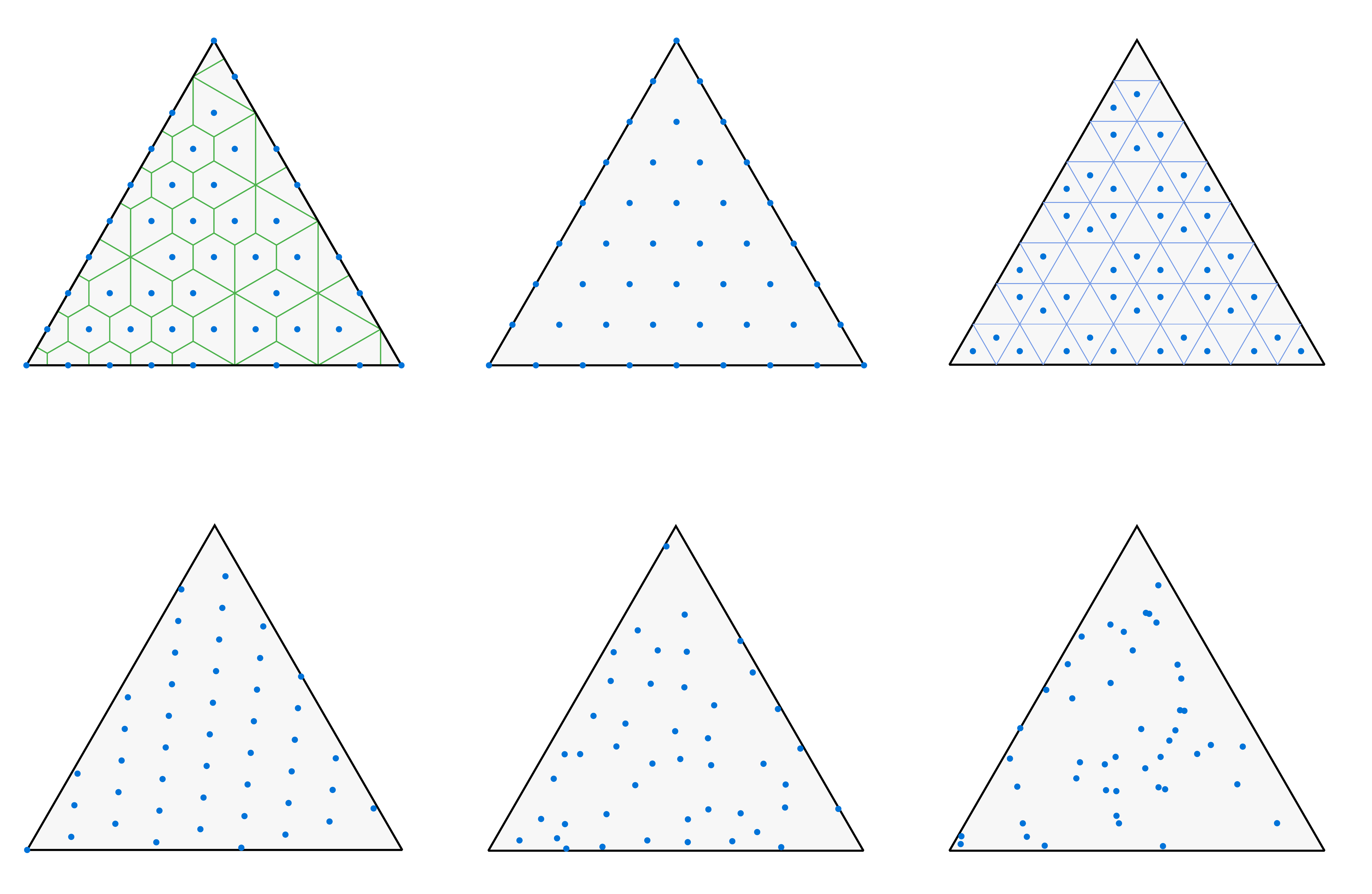}
    \caption{Six point sets with $n=45$: our VG algorithm (left top), barycentric grid (middle top), triangular van der Corput sequence (right top), triangular Kronecker lattice (left bottom), PD random (middle bottom), and i.i.d.\ random (right bottom).}
    \label{fig:4-3}
\end{figure}

We first conduct experiments on the unit equilateral triangle, defined with vertices at $(0,0)$, $(1,0)$, and $(1/2, \sqrt{3}/2)$. 
Figure~\ref{fig:4-3} provides a visual comparison of the point sets generated by all six methods for $n=45$ on this domain. 

We then compute the mesh ratio $\rho(P_n; T)$ incrementally for the first 210 points, starting from $n=3$. The results for all six methods are plotted in Figure~\ref{fig:4-31}. Across the entire range of $n$, the VG algorithm and the barycentric grid consistently outperform the other methods, with their mesh ratio curves remaining below the optimal bound of $\rho \le 2$. Moreover, when $n$ is a triangular number (i.e., the set forms a complete lattice), the grid achieves the theoretical minimum mesh ratio of $2/\sqrt{3}$ (for an equilateral triangulation). Notably, the VG algorithm also reaches this lower bound for several $n$. This indicates that the VG algorithm effectively suppresses local voids and maintains a near-optimal quasi-uniformity, demonstrating its robustness for all values of $n$.

\begin{figure}[t]
   \centering
    \includegraphics[width=\linewidth]{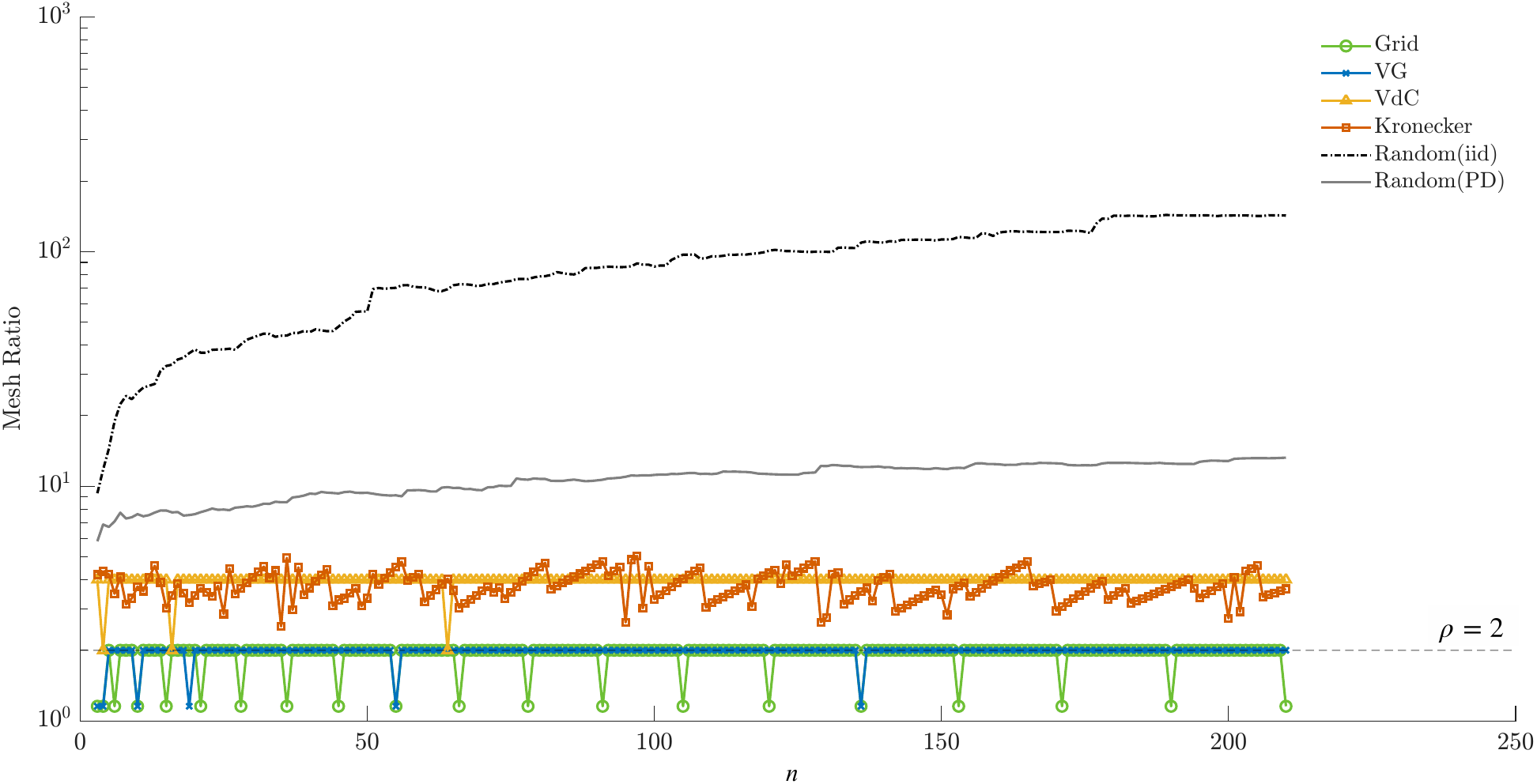}
    \caption{Mesh ratio of six point sets in the unit equilateral triangle: our VG algorithm (blue), barycentric grid (green), triangular van der Corput sequence (orange), triangular Kronecker lattice (red), PD random (black dash), and i.i.d.\ random (black).}
    \label{fig:4-31}
\end{figure}

The Kronecker lattice performs moderately but exhibits significant sawtooth-like fluctuations. This likely reflects the transient formation of local gaps as the rotated lattice points are clipped to the triangular domain. The van der Corput sequence is stable, but its mesh ratio is consistently high. This clearly illustrates that low-discrepancy properties do not automatically guarantee a low mesh ratio. Finally, the two random baselines perform the worst. The mesh ratio for the i.i.d.\ set appears to increase monotonically with $n$. The PD set, which enforces a minimum point distance, performs considerably better than i.i.d.\ but is still substantially worse than the other deterministic, constructive sets.

\begin{figure}[t]
   \centering
    \includegraphics[width=\linewidth]{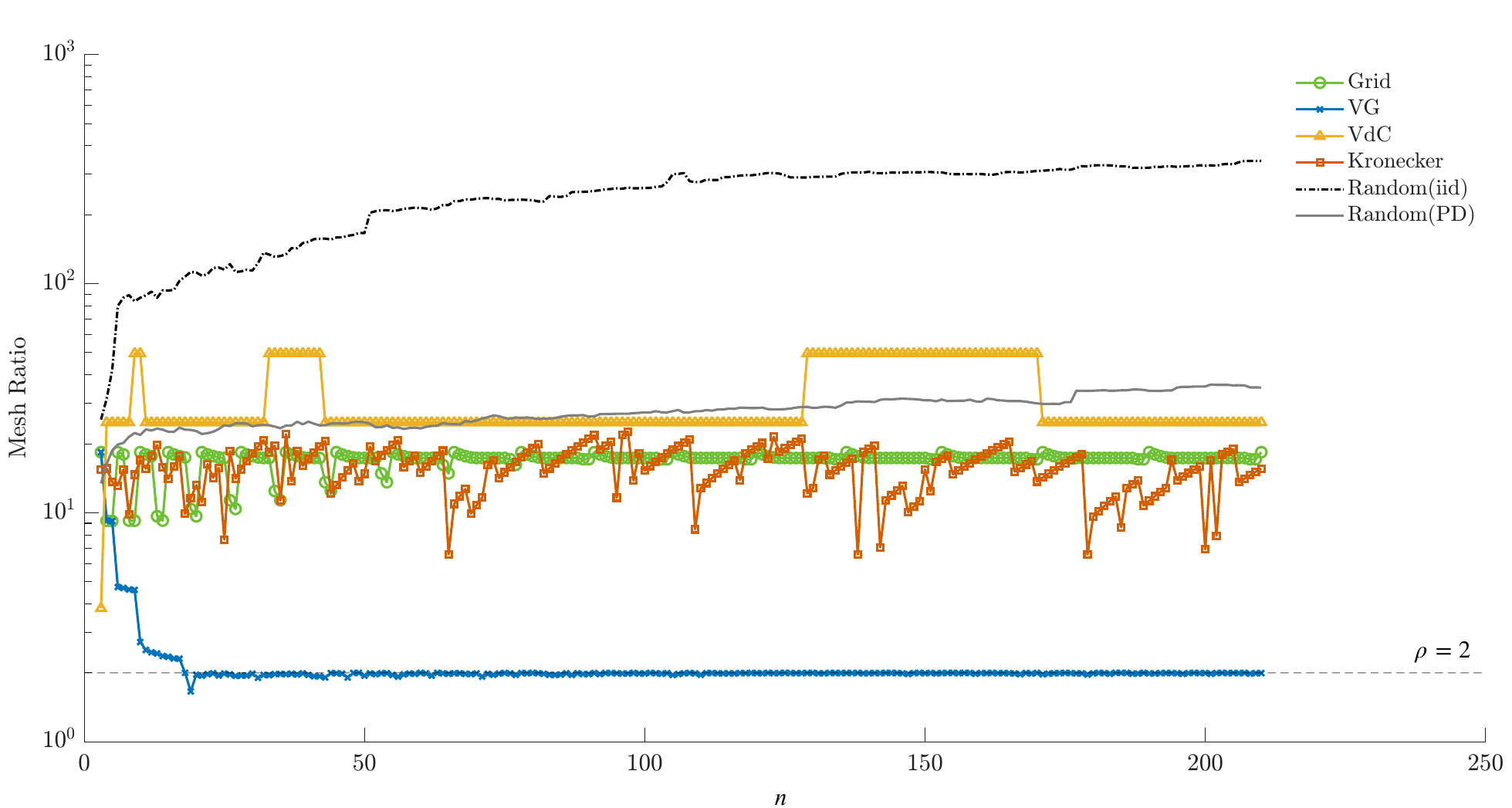}
    \caption{Mesh ratio of six point sets in a skinny triangle with vertices $(0,0), (1,0), (0.028,0.045)$: our VG algorithm (blue), barycentric grid (green), triangular van der Corput sequence (orange), triangular Kronecker lattice (red), PD random (black dash), and i.i.d.\ random (black).}
    \label{fig:4-32}
\end{figure}

To test the robustness of the different constructions under severe geometric degeneracy, we next conduct the same mesh ratio comparison on a skinny triangle, defined with vertices at $(0,0), (1,0),$ and $(0.028, 0.045)$. This triangle with its isoperimetric quotient $J=12\sqrt{3}\,A/L^2 \approx 0.114$ has one sharp angle ($\approx 2.6^\circ$) and a large obtuse angle ($\approx 119.3^\circ$), representing a strongly degenerate, asymmetric case. The results for all six methods are plotted in Figure~\ref{fig:4-32}. In this challenging domain, the VG algorithm demonstrates clear superiority, consistently outperforming all other methods. After an initial period, its mesh ratio successfully converges to and remains at the optimal bound ($\rho \le 2$). The initial, monotonic decrease of the mesh ratio (until it reaches $2$) is consistent with the theoretical prediction from the proof of the third item of Lemma~\ref{lem:dynamics}. In contrast, the mesh ratios for the barycentric grid and the van der Corput sequence are stable but at magnitudes significantly higher than the VG algorithm. The Kronecker lattice performs comparably to the barycentric grid but, as before, exhibits significant sawtooth-like fluctuations. The i.i.d.\ random baseline performs the worst, with its mesh ratio increasing almost monotonically with $n$. The PD random set performs comparably to the van der Corput sequence. However, its mesh ratio increases gradually, suggesting it will perform worse than the deterministic, constructive sets for larger $n$. This is consistent with the fact that the PD construction mainly improves the separation radius by preventing clustering, but does not explicitly minimize the covering radius; consequently, relatively large local gaps may remain, leading to a larger mesh ratio.

\subsection{RBF interpolation performance}

Finally, we study how different point sets affect the accuracy and stability of 2-D RBF interpolation on a triangular domain. Both the domain and the test functions are first normalized. All point sets are generated on the unit equilateral triangle embedded in $[0,1]^2$ with corners at $(0,0)$, $(1,0)$, and $(1/2, \sqrt{3}/2)$. 
For this experiment, the grid point set is constructed using the barycentric grid corresponding to a triangular number. 
It serves as a mesh-based baseline, achieving the smallest mesh ratio of $2/\sqrt{3}$ when $n$ is a triangular number, as shown in Figure~\ref{fig:4-31}, and is expected to achieve near-optimal asymptotic rates when the kernel supports are appropriately chosen. It is important to note, however, that this desirable property of the grid does not hold for ill-conditioned triangles, see Figure~\ref{fig:4-32}.

The test functions and their targeted properties are listed in Table~\ref{tab:test_functions}. Using the six point sets described in the previous subsection, we perform RBF interpolation. Given nodes $\{\bsx_i\}_{i=1}^n$, a radial kernel $\phi$, and a target function $f$, the interpolant is defined as
\[
S(\bsx)=\sum_{i=1}^{n} w_i\,\phi\left(\|\bsx-\bsx_i\|\right),
\]
where the weights $\{w_i\}$ are determined by the usual linear system with entries $\phi\left(\|\bsx_i-\bsx_j\|\right)$. We consider the following three kernels:
\begin{itemize}
\item Gaussian kernel: 
\[ \phi(r) = e^{-(r/\ell)^2} \]
\item Mat\'{e}rn--5/2 kernel: 
\[ \phi(r) = \left(1 + \sqrt{5}\frac{r}{\ell} + \frac{5}{3}\left(\frac{r}{\ell}\right)^2 \right) e^{-\sqrt{5}\, r / \ell} \] 
\item Wendland--$C^2$ kernel: 
\[ \phi(r) = \left(1-\frac{r}{\ell}\right)_+^4 \left(4\,\frac{r}{\ell}+1\right), \]
where $(x)_+ := \max(x,0)$ for $x \in \mathbb{R}$.
\end{itemize}
The kernel parameter $\ell$, common to all three kernels, is set as $\ell = c \sqrt{A/n}$, where $A$ is the area of the normalized triangle, $n$ is the number of nodes, and $c>0$ is a user-controlled coefficient. We assess the accuracy on a dense validation grid within the triangle and report the root-mean-square error, denoted by $E_2$, to evaluate convergence.

\begin{table}[t]
\centering
\caption{Test functions used for the numerical experiments.}
\label{tab:test_functions}
\renewcommand{\arraystretch}{2.0} 
\begin{tabular}{l l p{5cm}}
\hline\hline
\textbf{Identifier} & \textbf{Mathematical Expression} & \textbf{Key Properties} \\
\hline
\texttt{Franke} & Standard Franke's function$^*$  & A globally smooth surface composed of several Gaussian-like peaks and dips. \\
\texttt{Fourier2d} & $f(x,y) = \sin(9\pi x) \cos(9\pi y)$ & A smooth, highly oscillatory function with a regular wave-like pattern. \\
\texttt{ridge} & $f(x,y) = \displaystyle \frac{\arctan\bigl(2(x + 3y - 1)\bigr)}{\arctan\bigl(2(\sqrt{10}+1)\bigr)}$ & Features a steep, diagonally-oriented gradient layer along the line $x+3y=1$ \cite{LF03}. \\
\texttt{runge} & $f(x,y) = \displaystyle \frac{25}{25 + (x - 0.2)^2 + 2y^2}$ & An anisotropic 2D Runge-like function with a sharp, elliptical peak centered at $(0.2, 0)$ \cite{LF03}. \\
\hline\hline
\end{tabular}
\begin{align*}
    ^*:\quad f(x,y) & = \frac{3}{4} e^{-((9x-2)^2+(9y-2)^2)/4} + \frac{3}{4} e^{-((9x+1)^2/49 + (9y+1)/10)}\\
    & \quad + \frac{1}{2} e^{-((9x-7)^2 + (9y-3)^2)/4} - \frac{1}{5} e^{-((9x-4)^2 + (9y-7)^2)}
\end{align*}
\vspace{-1ex}
\hrule
\end{table}

\begin{figure}[t]
   \centering
    \includegraphics[width=\linewidth]{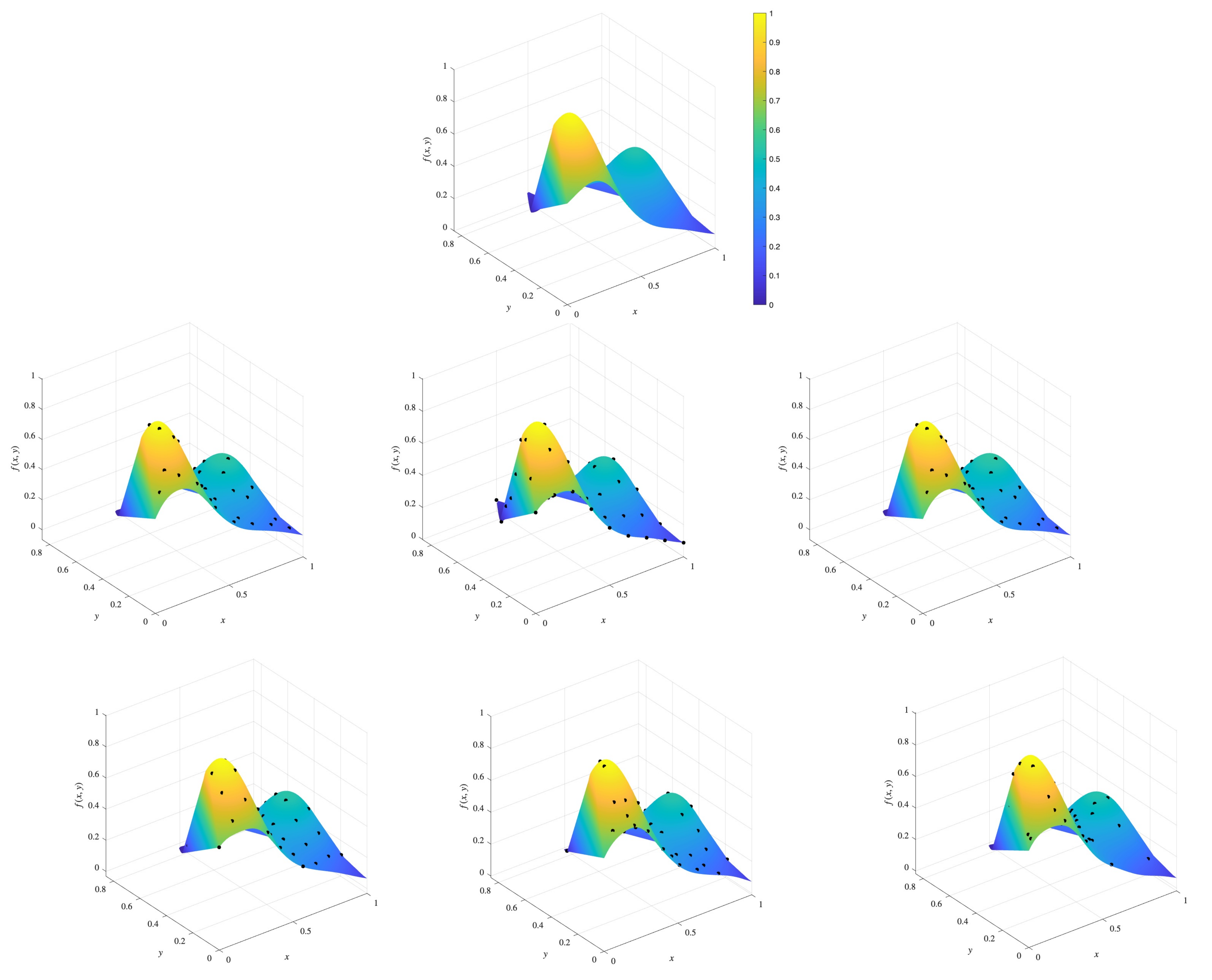}
    \caption{RBF interpolation of the \texttt{Franke} function (Gaussian kernel, $n=45$): true surface (top), our VG algorithm (middle left), barycentric grid (middle center), triangular van der Corput sequence (middle right), triangular Kronecker lattice (bottom left), PD random (bottom center), and i.i.d.\ random (bottom right).}
    \label{fig:4-4}
\end{figure}

\begin{figure}[!]
   \centering
    \includegraphics[width=0.85\textwidth]{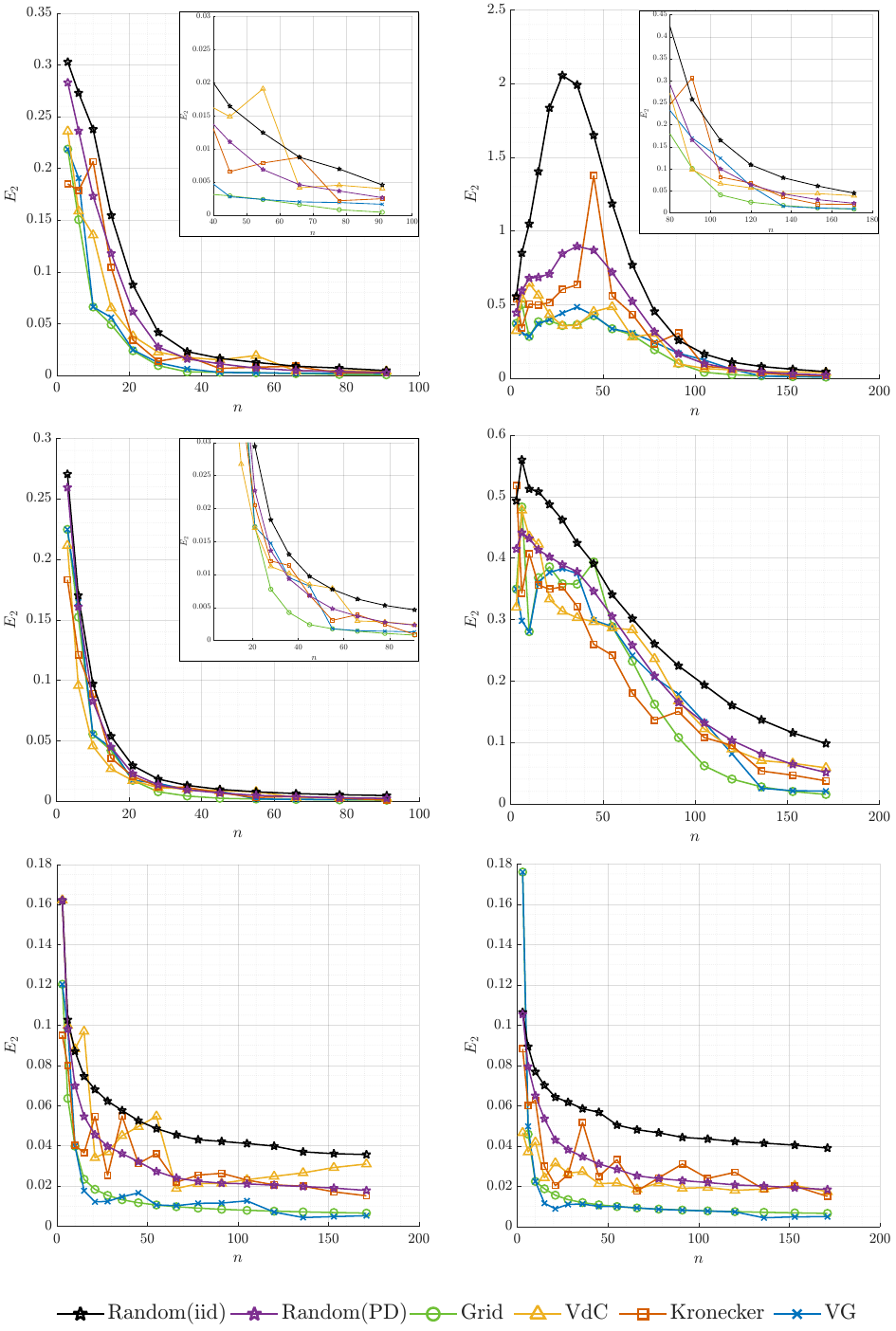}
    \caption{Convergence behavior of the RBF interpolation RMSE ($E_2$) for various test functions and kernels: \texttt{Franke} with Gaussian kernel (top left), \texttt{Fourier2d} with Gaussian kernel (top right), \texttt{Franke} with Mat\'{e}rn--5/2 kernel (middle left), \texttt{Fourier2d} with Mat\'{e}rn--5/2 kernel (middle right), \texttt{ridge} with Wendland-$C^2$ kernel (bottom left), and \texttt{runge} with Wendland-$C^2$ kernel (bottom right). Different colors represent different point sets used. The inset plots provide zoomed-in views of regions with smaller errors to highlight performance differences.}
    \label{fig:4-5}
\end{figure}

Figure~\ref{fig:4-4} visualizes the RBF interpolation of the \texttt{Franke} function using the Gaussian kernel for $n=45$. While all point sets apparently produce reasonable results, a closer inspection reveals that our VG and grid point sets reproduce the true surface most accurately. The other methods, in contrast, slightly fail to capture certain local details. This observation will be quantified with the $E_2$ error.

The convergence behavior of the $E_2$ error (as a function of the number of points $n$) for various combinations of the test function and the kernel is shown in Figure~\ref{fig:4-5}. For these experiments, the user-controlled coefficient $c$ was set as follows: $c=2$ in the top-right and middle-right plots, $c=4$ in the top-left and middle-left plots, and $c=5$ in the bottom-left and -right plots. Obviously, different values of $c$ can lead to different absolute error decays, sometimes producing extremely slow or unstable convergence. However, the relative performance ranking among the six point sets remains largely consistent in most cases.

As shown in the figure, the barycentric grid (the optimal baseline for an equilateral triangle) performs well, as expected. Our VG algorithm is asymptotically comparable to the barycentric grid in all cases, validating the robustness of quasi-uniform point sets with a small mesh ratio. The van der Corput sequence, the Kronecker lattice, and the PD random set perform moderately. They show some exceptions for small $n$, where they perform better than the barycentric grid. However, as $n$ increases, the barycentric grid and our VG algorithm consistently overtake these sets in all test cases. Finally, the i.i.d.\ random baseline performs the worst, exhibiting the slowest convergence.

\bibliographystyle{amsplain}
\bibliography{ref.bib}

\end{document}